\documentclass[12pt,reqno]{amsproc}
\usepackage[margin=0.9in]{geometry}
\usepackage[utf8]{inputenc}
\usepackage{color}
\usepackage{tikz-cd}
\usepackage[hidelinks]{hyperref}
\usepackage[shortlabels]{enumitem}
\usepackage{theoremref}
\usepackage{ amssymb, amsthm, amsmath, bbm, mathtools }
\usepackage{braket}
\input xy
\xyoption{all}
\usepackage{soul}

\numberwithin{equation}{section}

\newtheorem{theorem}{Theorem}[section]
\newtheorem{corollary}[theorem]{Corollary} 
\newtheorem{lemma}[theorem]{Lemma}
\newtheorem{proposition}[theorem]{Proposition}
\newtheorem{question}{Question}

\newtheorem{convention}[theorem]{Convention}

\theoremstyle{definition}
\newtheorem{example}[theorem]{Example}
\newtheorem{definition}[theorem]{Definition}
\newtheorem{remark}[theorem]{Remark}
\newtheorem{co  on nvention}[theorem]{Convention}

\newtheorem{construction}{Construction}

\newcommand{\CC}{{\mathbb{C}}}
\newcommand{\RR}{{\mathbb{R}}}
\newcommand{\KK}{{\mathbb{K}}}

\newcommand{\PP}{{\mathbb{P}}}
\newcommand{\Mcal}{{\mathcal{M}}}
\newcommand{\Gcal}{{\mathcal{G}}}
\newcommand{\T}{^\mathsf{T}}

\newcommand{\GL}{\operatorname{GL}}
\newcommand{\adj}{\operatorname{adj}}
\newcommand{\coker}{\operatorname{coker}}

\newcommand{\Mat}{\operatorname{Mat}}

\newcommand{\Hom}{\operatorname{Hom}}
\newcommand{\im}{\operatorname{im}}

\newcommand{\pa}{\mathrm{pa}}

\newcommand{\PD}{{\rm PD}}

\DeclareMathOperator{\mlt}{mlt}

\definecolor{forest}{RGB}{11,128,35}

\definecolor{violet}{rgb}{0.54, 0.17, 0.89}

\title[Complete collineations for maximum likelihood estimation]{Complete collineations \\ for maximum likelihood estimation}
\date{}
\author{Gergely Bérczi}
\address{Aarhus University}
\email{gergely.berczi@math.au.dk}
\author{Eloise Hamilton}
\address{University of Cambridge}
\email{eloise.hamilton@newn.cam.ac.uk} 
\author{Philipp Reichenbach}
\address{Technische Universit\"at Berlin}
\email{reichenbach@tu-berlin.de}
\author{Anna Seigal}
\address{Harvard University}
\email{aseigal@seas.harvard.edu}

\begin{document}

\maketitle

\begin{abstract}
We import the algebro-geometric notion of a complete collineation into the study of maximum likelihood estimation in directed Gaussian graphical models. 
A complete collineation produces a perturbation of sample data, which we call a stabilisation of the sample. 
While a maximum likelihood estimate (MLE) may not exist or be unique given sample data, it is always unique given a stabilisation. 
We relate the MLE given a stabilisation to the MLE given original sample data, when one exists, providing
necessary and sufficient conditions for the MLE given a stabilisation to be one given the original sample.
For linear regression models, we show that the MLE given any stabilisation is the minimal norm choice among the MLEs given an original sample.
We show that the MLE has a well-defined limit as the  stabilisation of a sample tends to the original sample, and that the limit is an MLE given the original sample, when one exists. 
Finally, we study which MLEs given a sample can arise as such limits.
We reduce this
to a question regarding the non-emptiness of certain algebraic varieties.  
\end{abstract} 

%\tableofcontents

\section{Introduction}

We study maximum likelihood estimation in directed Gaussian graphical models.
The existence or uniqueness of a maximum likelihood estimate (MLE) given observed data is known to depend on the number of samples and their genericity~\cite{buhl1993existence,drton2019maximum,gross2014maximum}. 
Several approaches have been proposed to compute MLEs given data that is insufficient or non-generic, including regularisation~\cite{danaher2014joint}, dividing the model into sub-networks~\cite{wille2004sparse}, and reducing the number of parameters via symmetries~\cite{makam2021symmetries}.
We propose a new approach to this problem based on the algebro-geometric concept of a complete collineation. The idea is that if an MLE does not exist or is not unique given observed data, the data may be perturbed so that a unique MLE can be found. This unique MLE can then be used to single out an MLE given the initial data, if one exists, or otherwise to obtain a statistically meaningful MLE given the initial data. The key to proving these results is to have the right notion of perturbation. We propose that perturbations arising from complete collineations are a natural choice.

The distributions we consider are mean-centred $m$-dimensional Gaussians, for some dimension $m$.
Our models are parametrised by certain subsets of the cone of $m \times m$ positive definite matrices.
Sample data can be collected into a matrix $Y$ of size $n \times m$, where $n$ is the number of observations. The existence or uniqueness of the MLE given $Y$ depends on the model and on the properties of the matrix $Y$. For example, if the model is the full cone of positive definite matrices, the MLE given $Y$ exists and is unique if and only if $Y$ has full column rank. This cannot occur for $n < m$ but occurs generically once $n \geq m$.

In this paper, we think of a sample $Y \in \RR^{n \times m}$ as a linear map $\RR^m \to \RR^n$. If the MLE does not exist or is not unique given $Y$, then $Y$ is a degenerate linear map, i.e.\ it does not have maximal rank. We adopt the view that $Y$ should be considered not on its own but together with the additional information that a complete collineation provides. This additional information can be packaged into a new sample $\widetilde{Y}$, which we call a stabilisation of $Y$, such that the MLE is unique given $\widetilde{Y}$ and such that this MLE can be related to MLEs given $Y$ if one exists. The extra information carried by $\widetilde{Y}$ should be thought of as ensuring $\widetilde{Y}$ is a `well-behaved' degeneration of a sample corresponding to a non-degenerate linear map, i.e.\ a map of maximal rank. In particular, the sample $\widetilde{Y}$ should be viewed as a `better' degeneration than the degenerate map $Y$ itself.  \\

\textbf{Why complete collineations?} 
While degenerate linear maps are the most obvious candidates for degenerations of non-degenerate linear maps, an important lesson originating in the work of late 19th century geometers is that they are not the right notion of degeneration from the point of view of enumerative geometry~\cite{Kleinman1988}. The key insight from this line of work is that degenerations should carry more information than just that of a degenerate linear map; the key contribution lies in identifying exactly what this information should be. Complete collineations encode the necessary information. 

A collineation between two projective spaces $\PP(V)$ and $\PP(W)$ is the scalar equivalence class $[f]$ of a non-degenerate linear map $f: V \to W$. By convention, we map from the smaller projective space to the larger one, so we assume that $\dim V \leq \dim W$. The term collineation originates in the fact that the map $[f]$ sends collinear points in $\PP(V)$ to collinear points in $\PP(\im f)$ bijectively. 
In fact, the map $[f]$ not only maps lines to lines but also maps $i$-planes to $i$-planes, via associated maps $[\wedge^i f]: \PP(\wedge^i V)\to \PP(\wedge^i W)$, for $i$ from $1$ to $\dim V$. By contrast, if $f$ is a degenerate map from $V$ to $W$, then while the equivalence class $[f]$ is well-defined, the equivalence classes $[\wedge^i f]$ may no longer be well-defined, since $i$-planes may collapse to $j$-planes for some $j < i$. 
Complete collineations are degenerations of collineations that preserve the higher-order information of the $i$-plane to $i$-plane correspondences $[\wedge^i f]$~\cite[p254]{Kleinman1988}.
Concretely, a complete collineation from $\PP(V)$ to $\PP(W)$ with $\dim V \leq \dim W$ is a finite sequence $([f_1],\hdots, [f_t])$ of equivalence classes of linear maps, where $f_1: V \to W$, $f_i: \operatorname{ker} f_{i-1} \to \operatorname{coker} f_{i-1}$ for $i \geq 2$ and $f_t$ is the first non-degenerate map -- see Section \ref{sec:agprelim} for details.

Given a sample corresponding to a degenerate linear map $f: \RR^m \to \RR^n$, we define a stabilisation of $f$, or $f$-stabilisation, to be a sample $\widetilde{f} : = f + f'$ where the perturbation $f'$ comes from a complete collineation between $\PP(\RR^m)$ and $\PP(\RR^n)$ with first term $[f]$. We will always reduce to the case where $m \leq n$ (see Section \ref{sec:relating}). A precise definition of sample stabilisations is in Section \ref{subsec:samplestab}. Properties of complete collineations ensure that the MLE is unique given $\widetilde{f}$.  While various conditions could be placed on $f'$ to ensure that the MLE given $\widetilde{f}$ is unique, our conditions have the advantage that the MLE given $\widetilde{f}$ and the MLEs given $f$, if they exist, are closely related. In this paper, we use complete collineations to resolve non-identifiability of the MLE.\\

\textbf{Main results.} 
Fix a directed acyclic graph (DAG) $\Gcal$ on vertices $\{ 1, \ldots, m \}$ with edge set $E$. A directed edge from $j$ to $i$ is denoted by $j \to i$. The acyclicity rules out directed cycles $j \to i \to \cdots \to k \to j$. A child vertex is a vertex $i$ with a parent in $\Gcal$, i.e.\ with an edge $j \to i$ in $\Gcal$ for some vertex $j$. 
The statistical models we consider are directed Gaussian graphical models on DAGs. We call these DAG models, for short.
They have $m + |E|$ parameters, one for each edge and one for each vertex.
The MLE given data $Y$ consists of estimates for all of these parameters -- see Section~\ref{sec:asprelim} for details.
We work throughout over a field $\KK$ which can be taken to be either $\RR$ or $\CC$. Our results hold over both fields.

Our first main result relates the MLE given a stabilisation to an MLE given an original sample. 
We denote the span of a set of vectors $\{ v_1, \ldots, v_k\}$ by $\langle v_1, \ldots, v_k \rangle$ and the projection of a vector $v$ onto a linear space $L$ by $\pi_L(v)$.

\begin{theorem} \thlabel{firstmainresult}
Fix a DAG $\mathcal{G}$ and a sample $f \in \KK^{n \times m}$. Let $\widetilde{f} = f+f'$ denote a stabilisation of $f$. Let $f_i$ and  $v_i$ denote the columns of $f$ and $f'$ respectively. We have the following results concerning maximum likelihood estimation in the DAG model on $\Gcal$: \begin{enumerate}[(a)]
\item the MLE given $\widetilde{f}$ is unique; \label{unique} 
\item the MLE given $\widetilde{f}$ is an MLE given $f$ if and only if for all child vertices $i$ in $\mathcal{G}$ we have: 
$$ v_i \in \langle v_j :j \to i \rangle \text{ and }  \overline{f_i} +  v_i \in \langle f_j + v_j :j \to i \rangle, $$ where $\overline{f_i} := \pi_{\langle f_j   : j \to i \rangle}(f_i)$; \label{iff} 
\item the MLE given $\widetilde{f}(\epsilon) : = f + \epsilon f'$ is unique for all $\epsilon \neq 0$, and has a well-defined limit as $\epsilon$ tends to zero, called the \emph{limit MLE given $\widetilde{f}$}, which can be described explicitly (see \thref{mainresult}\ref{partc}); \label{limitexists}
\item the limit MLE given $\widetilde{f}$ is an MLE given $f$, if one exists. \label{limitis} 
\end{enumerate} 
\end{theorem} 

Our second main result addresses when an MLE given $f$ is the MLE or limit MLE given a stabilisation of $f$.  

\begin{theorem} \thlabel{secondmainresult} 
Fix a DAG $\mathcal{G}$ and a sample $f \in \KK^{n \times m}$. Let $\alpha$ denote an MLE given $f$ in the DAG model on $\Gcal$. Then: 
\begin{enumerate}[(a)]
\item there is a locally closed subvariety $X_{f}  \subseteq \KK^{n \times m}$ parametrising stabilisations of $f$; \label{psforstab} 
\item there is a closed subvariety $X_{f, \alpha} \subseteq X_f,$ with  defining equations given in \eqref{explicit}, parameterising stabilisations $\widetilde{f}$ of $f$ such that the MLE given $\widetilde{f}$ is $\alpha$, so that $$X_{f,\alpha} \neq \emptyset$$ if and only if $\alpha$ is the MLE given an $f$-stabilisation; \label{ps2} 
\item there is a closed subvariety $X_{f, \alpha}^{\operatorname{lim}} \subseteq X_f,$ with  defining equations given in \eqref{fixedlambda}, parameterising stabilisations $\widetilde{f}$ of $f$ such that the limit MLE given $\widetilde{f}$ is $\alpha$, so that $$X_{f,\alpha}^{\operatorname{lim}} \neq \emptyset$$ if and only if $\alpha$ is the limit MLE given an $f$-stabilisation. \label{ps3} 
\end{enumerate} 
\end{theorem}

We apply the above results to DAG models on a star-shaped graph, which in this paper refers to a connected DAG with a unique child vertex, as in Figure~\ref{fig:star}. 
Such models study the linear dependence of one variable on all others -- they are linear regression models with Gaussian noise.

\begin{figure}[htbp]
\begin{tikzpicture}[->,every node/.style={circle,draw},line width=1pt, node distance=1.7cm]
	
	% Central vertex
	\node (n)  {$5$};
	
	% Surrounding vertices
	\node (1) [above left of=n] {$1$};
	\node (2) [below left of=n] {$2$};
	\node (3) [below right of=n] {$3$};
	\node (4) [above right of=n] {$4$};
	
	% Arrows to the central vertex (n)
	\foreach \to in {1,2,3,4}
	\draw (\to) -- (n);
\end{tikzpicture}
\caption{Star-shaped graph with $m=5$ vertices.}
\label{fig:star}
\end{figure}
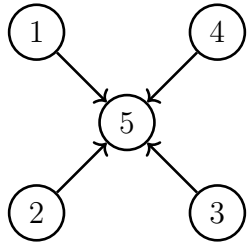

\begin{theorem} \thlabel{linearregression} 
Consider a star-shaped DAG $\mathcal{G}$ and a sample $f$. If the MLE given $f$ exists in the DAG model on $\Gcal$, then the MLE given any stabilisation of $f$ is 
the unique MLE given $f$ of minimal $2$-norm.
\end{theorem} 

\thref{linearregression}
exhibits a model where exactly one of the MLEs given a sample $f$ can be obtained from the MLE or limit MLE given a stabilisation of $f$.
The unique MLE singled out should be viewed as the `preferred' one, thus resolving the problem of non-identifiability of the MLE given $f$. 
For other DAGs, different stabilisations may give different MLEs, and resolving non-identifiability of the MLE relies on a choice of stabilisation. We describe in Section \ref{subsec:ccassamples} a sampling algorithm which constructs an $f$-stabilisation from any sample $f$ via a finite sequence of samples. Each sample is obtained by sampling linear combinations of the nodes of $\mathcal{G}$, with the number of samples needed strictly decreasing at each step. \\

\textbf{Related work.}  
We are not aware of any existing work connecting complete collineations to algebraic statistics. Nevertheless, in a different direction the closely related concept of complete quadrics has recently been used in algebraic statistics to study particular classes of Gaussian models \cite{michalek2021maximum,dinu2021geometry, manivel2023complete}. Complete quadrics are defined analogously to complete collineations, with the additional constraint that $\dim V = \dim W$ and that $f$ is symmetric; their moduli space enjoys the same features as the moduli space of complete collineation. \cite{michalek2021maximum, manivel2023complete} study generic linear concentration models. These are Gaussian models whose concentration matrices (i.e.\ the inverses of the covariance matrices) are $m \times m$ positive definite matrices lying in a fixed $d$-dimensional generic linear subspace $L$ of $m \times m$ symmetric matrices. The ML degree of such a model is the number of complex critical points of the log-likelihood function for a generic sample covariance matrix, which depends only on $m$ and $d$ by genericity of $L$ and is denoted by $\phi(m,d)$. \cite{michalek2021maximum, manivel2023complete}  connect intersection theory on the space of complete quadrics to the computation of $\phi(m,d)$, leading to a proof that $\phi(m,d)$ is polynomial in $m$ for  fixed $d > 0$ in \cite{manivel2023complete}, as conjectured in \cite{sturmfels2010multivariate}. By contrast, \cite{dinu2021geometry} considers Gaussian graphical models, which are examples of non-generic linear concentration models, and uses intersection theory on the space of complete quadrics to compute the degree of the projective variety associated to Gaussian graphical models on cyclical graphs, answering another conjecture of~\cite{sturmfels2010multivariate}.\\

\textbf{Organisation.} We give preliminaries from algebraic geometry and algebraic statistics in 
Sections \ref{sec:agprelim} and \ref{sec:asprelim} respectively.  We review, for different DAG models, which MLE properties can occur in Section \ref{sec:classification}. We introduce sample stabilisations and their parameter spaces in Section \ref{sec:fromcctoss} (\thref{secondmainresult}\ref{psforstab}), and show how a sample stabilisation is constructed from a complete collineation. Sections \ref{sec:MLEofsamplestab}--\ref{sec:linearregression} prove the main results. Section \ref{sec:MLEofsamplestab} focuses on the MLE given a sample stabilisation (\thref{firstmainresult}\ref{unique} and \ref{iff}, and \thref{secondmainresult}\ref{ps2}). Section \ref{sec:limitsolutions} constructs unique solutions to underdetermined linear systems as the limit of a solution to a perturbation of the linear system. This result is applied in Section \ref{sec:uniqueMLEinlimit} to study the limit MLE given sample stabilisations (\thref{firstmainresult}\ref{limitexists} and \ref{limitis}, and \thref{secondmainresult}\ref{ps3}). In Section \ref{sec:linearregression} we apply the results to linear regression models (\thref{linearregression}). Finally we discuss directions for future work in Section \ref{sec:outlook}.\\

\textbf{Acknowledgments:} We thank Visu Makam for helpful discussions. EH thanks Johan Martens for bringing complete collineations and \cite{vainsencher1984complete} to her attention in a different context, and Dhruv Ranganathan for useful discussions. PR acknowledges funding by the European Research Council (ERC) under the European’s Horizon 2020 research and innovation programme (grant agreement no. 787840).
AS was supported by the NSF (DMR-2011754), and GB was supported by Aarhus University Starting Grant AUFF-29289.

\section{Algebraic Geometry preliminaries} \label{sec:agprelim}

We review the construction of the moduli space of complete collineations and the definition of a complete collineation that we will work with in this paper.

\subsection{The moduli space of complete collineations}

We start by defining the moduli space of complete collineations. A complete collineation is an element of this moduli space. We will give another definition of a complete collineation that is easier to work with in Section \ref{subsec:pointsofms}. 

\begin{definition}[The moduli space of complete collineations] \label{def:spaceCompleteCollineations}
Fix two vectors spaces $V$ and $W$ with $\dim V \leq \dim W$. The \emph{moduli space of complete collineations from $\PP(V)$ to $\PP(W)$} is the closure of the graph of the rational map 
 \begin{align*} 
\phi: \PP(\Hom(V,W)) & \dashrightarrow \PP(\Hom( \wedge^2 V, \wedge^2 W)) \times \cdots \times \PP(\Hom( \wedge^r V, \wedge^r W))\\
[M] & \mapsto ([\wedge^2 M ], \hdots, [\wedge^r M]), 
\end{align*} where $r = \dim V$. 
\end{definition} 
Note that $\phi$ is only well-defined on the locus inside $\PP(\operatorname{Hom}(V,W))$ parametrising collineations, i.e.\ maps of maximal rank.

By construction, the moduli space of complete collineations contains as an open dense subset the space of maximal rank linear maps up to scaling. It can therefore be viewed as a compactification of the space of maps of maximal rank in $\PP(\operatorname{Hom}(V,W))$. This is an alternative compactification to the `obvious' one given by $\PP(\Hom(V,W))$, and has the advantage of having nicer geometric properties: its boundary is a normal crossing divisor, by contrast with the compactification given by $\PP(\Hom(V,W))$ whose boundary is highly singular. This geometric property makes the moduli space of complete collineations useful for tackling enumerative geometry problems related to linear maps \cite{Hervier1982,Thaddeuscc}.

\subsection{Points of the moduli space} \label{subsec:pointsofms}
Despite the simple construction of the moduli space of complete collineations, describing points in the boundary is difficult. In other words, given an element of $$\PP(\operatorname{Hom} (V , W)) \times \PP(\Hom( \wedge^2 V, \wedge^2 W)) \times \cdots \times \PP(\Hom( \wedge^r V, \wedge^r W))$$ with first term not of maximal rank, it is not obvious which properties the remaining terms need to satisfy for the element to lie in the moduli space of complete collineations. Thankfully, there is an alternative construction of the moduli space of complete collineations from which a description of points in the boundary can more readily be extracted.

This construction is obtained via a sequence of blow-ups of $\overline{C}: = \PP(\Hom(V,W))$, as shown by Vaisencher in \cite{vainsencher1984complete}. The sequence can be described inductively as follows: set $\overline{C}_0 = \overline{C}$ and for $i \geq 1$ let $\overline{C}_i$ denote the blow-up of $\overline{C}_{i-1}$ along the proper transform in $\overline{C}_{i-1}$ of the locus of maps $[f] \in \overline{C}$ of rank less than or equal to $i$. Then the moduli space of complete collineations from $V$ to $W$ is isomorphic to $\overline{C}_r$.
In particular, points of the blow-up are in one-to-one correspondence with complete collineations from $V$ to $W$. Moreover, points of the blow-up can be described explicitly by analysing the exceptional divisors at each stage of the blow-up. Doing so yields the following definition, which we will use for the rest of this paper.

\begin{definition}[Complete collineations]
Fix two vector spaces $V$ and $W$ with $\dim V \leq \dim W$. A \emph{complete collineation} from $\PP(V)$ to $\PP(W)$ is a finite sequence $([f_1],\hdots, [f_t])$ of scalar equivalence classes of maps: \begin{align*} 
f_1: V & \to W \\
f_2: \operatorname{ker} f_1 & \to \operatorname{coker} f_1 \\
& \vdots \\
f_t: \operatorname{ker} f_{t-1} & \to \operatorname{coker} f_{t-1} 
\end{align*} where each $f_i$ is degenerate except for $f_t$. An \emph{affine lift} of a complete collineation $([f_1],\hdots,[f_t])$ from $\PP(V)$ to $\PP(W)$ is a sequence $(g_1,\hdots, g_t)$ where $[g_i]=[f_i]$ for each $i$. 
\end{definition}

\section{Algebraic Statistics Preliminaries}  \label{sec:asprelim}

We give background on maximum likelihood estimation and DAG models. 

\subsection{Maximum likelihood estimation}

An $m$-dimensional Gaussian 
with mean zero has density
\begin{align*}
    f_\Sigma(y) = \frac{1}{\sqrt{\det( 2 \pi \Sigma)}}
    \exp \left( -\frac{1}{2} y\T \Sigma^{-1} y \right),
\end{align*}
where $y \in \RR^m$ and the covariance $\Sigma$ lies in the cone of $m \times m$ positive definite matrices  $\PD_m$. 
We refer to a multivariate Gaussian model by its set $\mathcal{M} \subset \PD_m$ of covariance matrices. The elements $\Sigma \in \mathcal{M}$ are parameters for the model. A maximum likelihood estimate (MLE) given sample data consists of parameters that maximise the likelihood of observing that sample.

We collect independent samples $Y_1, \ldots, Y_n \in \RR^m$ as the rows of a matrix $Y \in \RR^{n \times m}$.
Our convention that the rows are indexed by samples and the columns by variables is the transpose of that used in related work~\cite{makam2021symmetries,amendola2021invariant,PhDthesisPhilipp}.
A maximum likelihood estimate (MLE) given data $Y$ is a point $\hat{\Sigma} \in \Mcal$ that maximizes the likelihood of observing $Y$. 
The likelihood function is
$ L_Y(\Sigma) = \prod_{i=1}^n f_{\Sigma}(Y_i)$. 
We work with the 
function 
\begin{equation}\label{eqn:gaussianlikelihood}
    \ell_{Y} (\Sigma) = -\log \det (\Sigma) - \mathrm{tr} (\Sigma^{-1} S_Y),
\end{equation}
where $S_Y = \frac{1}{n} Y\T Y$.
This is the log-likelihood function, up to additive and positive multiplicative constants, hence has the same maximizers. 
An MLE given $Y$ in $\Mcal$ is therefore
$$ \hat{\Sigma} := \arg \max_{\Sigma \in \Mcal} \ell_Y (\Sigma), $$
if such a maximising $\Sigma \in \Mcal$ exists. 
We consider the following four properties which can occur when maximising $\ell_Y(\Sigma)$ over $\Sigma \in \Mcal$:
\begin{enumerate}
    \item[(a)] $\ell_Y$ is unbounded from above
    \item[(b)] $\ell_Y$ is bounded from above
    \item[(c)] the MLE exists (i.e.\ $\ell_Y$ is bounded from above and attains its supremum)
    \item[(d)] the MLE exists and is unique.
\end{enumerate}

\begin{example}
    \label{ex:whole_PD}
Let $\Mcal = \PD_m$ and fix a sample $Y \in \RR^{n \times m}$. The MLE given $Y$ is $S_Y = \frac{1}{n}Y\T Y$ if it is invertible, see e.g.~\cite[Proposition 5.3.7]{sullivant2018algebraic}. 
The matrix $S_Y$ lies in the model $\PD_m$ if and only if it is invertible.
If it is not invertible, then $\ell_Y$ is unbounded and the MLE does not exist.
Put differently, the MLE given $Y$ exists if and only if $Y$ has full column rank.
\end{example}

We define the \emph{maximum likelihood threshold} (mlt) of a multivariate Gaussian model to be the minimal number of samples needed for the MLE to generically exist and be unique. Example~\ref{ex:whole_PD} has $\mlt = m$. 
The study of maximum likelihood thresholds is an active area of study, with recent developments, including~\cite{bernstein2021maximum,drton2019maximum,gross2014maximum,blekherman2019maximum,drton2021existence,derksen2021maximum,derksen2022maximum}.

\begin{remark} \hfill
    \begin{itemize}
        \item[(i)] We assume that the mean is known to be zero. Alternatively, one could estimate the mean in addition to the covariance matrix, i.e.\ consider a model $\RR^m \times \Mcal$ with $\Mcal \subseteq \PD_m$. The MLE for the mean parameter is then the sample mean. Thus, after shifting to the sample mean one can  translate to the mean zero setting. This process shifts the maximum likelihood threshold by one, see \cite[Remark~6.3.7]{PhDthesisPhilipp}.
        
        \item[(ii)] For $m$-dimensional complex multivariate Gaussian distributions \cite{wooding1956multivariate}, one can do maximum likelihood estimation similarly to the above. The covariance matrix $\Sigma$ is Hermitian positive-definite and the sample matrix $Y$ lies in $\CC^{n \times m}$. The log-likelihood function is, up to additive and positive multiplicative constants, as in~\eqref{eqn:gaussianlikelihood} with $S_Y$ now formed using the conjugate transpose -- see \cite[Section~1.2]{derksen2021maximum} and \cite[Section~6.3]{PhDthesisPhilipp}. From here on we will work over $\KK \in \{\RR, \CC\}$, as in \cite{PhDthesisPhilipp}.
    \end{itemize}
\end{remark}

\subsection{Directed Gaussian graphical models} 

Linear structural equation models study linear relationships between noisy variables of interest. Directed Gaussian graphical models are a special case. 
    Let $\mathcal{G}=(V,E)$ be a DAG on vertices $V=\{1,2,\ldots, m\}$ and directed edges $E$.  
        A directed edge from $j$ to $i$ is denoted by $j \to i$ and the absence of such an edge by $j \not\to i$.
        The {\em parents} of $i$ in $\Gcal$ is the subset of vertices
$$ \pa(i) := \{ j \in V \mid (j \to i) \in E \}.$$ 
    A \emph{directed Gaussian graphical model on $\Gcal$} is defined by the linear structural equation
\begin{equation}
\label{eqn:lsem}
    y = \Lambda y + \varepsilon, \qquad \text{i.e.} \qquad y_i = \sum_{j \in \pa(i)} \lambda_{ij} y_j + \epsilon_i,
\end{equation} 
where $y \in \KK^m$, and $\lambda_{ij}=0$ for $j \not\to i$ in $\mathcal{G}$. 
Directed Gaussian graphical models assume normally distributed noise $\epsilon \sim N( \mu,\Omega)$ with $\Omega$ diagonal. We assume that the variables are mean-centred, so that $\mu = 0$.
The linear relationships are recorded in the term $\Lambda y$ while the noise term is $\epsilon$. 
We refer to a directed Gaussian graphical model on a DAG as a DAG model, for short.

The vector $y$ follows a multivariate normal distribution with mean $0$ and covariance 
\begin{equation}
    \label{eqn:cov}
    \Sigma = (I - \Lambda)^{-1} \Omega (I - \Lambda)^{-\ast}
\end{equation}  
by~\eqref{eqn:lsem}, where $\Lambda$ has entries $\lambda_{ij}$ and $(\cdot)^{-\ast}$ denotes inverse conjugate transpose (which is the inverse transpose if $\KK = \RR$).
The DAG model on $\Gcal$ is
$$ \Mcal = \{ \Sigma \in \PD_m \, \vert \, \Sigma = (I - \Lambda)^{-1} \Omega (I - \Lambda)^{-\ast}, \, \lambda_{ij} = 0 \text{ unless } j \to i \text{ in $\Gcal$, } \Omega \text{ diagonal} \} .$$ 
An MLE given $Y$ in the DAG model on $\Gcal$ consists of edge weights $\Lambda$ and variance $\Omega$.

    Denote the entries of $\Omega$ by $\omega_{i}$, and recall that $\lambda_{ij}$ are the entries of $\Lambda$. The function $\ell_Y$ from~\eqref{eqn:gaussianlikelihood} can be written in terms of the parameters $\omega_{i}$ and $\lambda_{ij}$. Its negation $- \ell_Y$ is
\begin{equation}
    \label{eqn:ly_lam_and_omega}
    \sum_{i=1}^m \left( \log \omega_{i} + \frac{1}{n \omega_{i}} \| Y^{(i)} - \sum_{j \in \pa(i)} \lambda_{ij} Y^{(j)} \|^2 \right),
\end{equation} 
where $Y^{(k)}$ denotes the $k$-th column of the sample matrix $Y$ for $k \in \{1, \hdots, m \}$, see ~\cite[Theorem 4.9]{makam2021symmetries}. 
An MLE given $Y$ consists of  $\hat{\lambda}_{ij}$ and 
$\hat{\omega}_{i}$ that minimize the above expression.
The $\hat{\lambda}_{ij}$ are therefore coefficients of each $Y^{(j)}$ in the orthogonal projection of $Y^{(i)}$ onto $\langle Y^{(j)} : j \in \pa(i) \rangle$. 
The $\hat{\omega}_{i}$ are the 
residuals 
$\frac{1}{n} \| Y^{(i)} - \sum_{j \in \pa(i)} \hat{\lambda}_{ij} Y^{(j)} \|^2$, provided that the residual is strictly positive -- see the proof of~\cite[Theorem 4.9]{makam2021symmetries} or of \cite[Theorem~6.3.16]{PhDthesisPhilipp}.

We can consider maximum likelihood estimation of just the $\Lambda$ parameters or just the $\Omega$ parameters. We refer to these as the $\Lambda$-MLE and $\Omega$-MLE given $Y$, respectively. 

\begin{example}
Let $\Gcal$ be the DAG $1 \to 3 \leftarrow 2$. The DAG model on $\Gcal$ is parametrised by $\lambda = (\lambda_{31}, \lambda_{32})$ and $\omega = ( \omega_{1}, 
\omega_{2}, 
\omega_{3})$.
Fix sample matrices
$$ 
Y = \begin{pmatrix} 1 & 0 \\ 0 & 1 \\ 
1 & 1 \end{pmatrix}, \qquad 
Y' = \begin{pmatrix} 1 & 0 \\ 1 & 0 \\ 
0 & 1 \end{pmatrix}, \qquad 
Y'' = \begin{pmatrix} 1 & 0 & 0 \\ 0 & 1 & 0 \\ 
0 & 0 & 1 \end{pmatrix}.$$
The $\Lambda$-MLE given $Y$ is $( 1 , 1 )$ and the $\Omega$-MLE given $Y$ does not exist. Hence the MLE given $Y$ does not exist. The $\Omega$-MLE given $Y'$ is $(\frac12, \frac12, \frac12)$, while the $\Lambda$-MLEs are $\{ (t,-t) : t \in \KK \}$.
Finally, the $\Omega$-MLE given $Y''$ is are $(\frac13, \frac13, \frac13)$ and the $\Lambda$-MLEs are
$\alpha = (0,0)$. 
\end{example}

 \begin{proposition}
 \label{prop:lambda_and_omega}
    A $\Lambda$-MLE always exists, but may not be unique. An $\Omega$-MLE may not exist, but is unique whenever it does.   
 \end{proposition}

 \begin{proof}
 Coefficients of each $Y^{(j)}$ in the projection of $Y^{(i)}$ onto $\langle Y^{(j)} : j \in \pa(i)\rangle$ always exist, hence a $\Lambda$-MLE always exists. 
The $\hat{\lambda}_{ij}$ are unique if and only if the submatrix of $Y$ with columns indexed by $\pa(i)$ has full colun rank.
The above residual formula for $\hat{\omega}_{i}$ shows that they are unique whenever they exist.
 \end{proof}

The existence and uniqueness of the MLE given a sample matrix $Y \in \KK^{n \times m}$ can be described by linear dependence conditions on $Y$. For a vertex $i$ in $\Gcal$ we write 
$Y^{(\pa(i))}$ for the sub-matrix of $Y$ with columns indexed by the parents of $i$ in $\Gcal$, and by $Y^{(\pa(i) \cup i)}$ the sub-matrix of $Y$ with columns indexed by $\{i\} \cup \pa(i)$.

\begin{theorem}[{See~\cite[Theorem 4.9]{makam2021symmetries} and \cite[Theorem~6.3.16]{PhDthesisPhilipp}}]
\thlabel{thm:DAG_Ymatrix}
Consider the DAG model on $\Gcal$ with $m$ vertices, and fix a sample matrix $Y \in \KK^{n \times m}$. The following possibilities characterise maximum likelihood estimation given $Y$:
   \[ \begin{matrix} \text{(a)} & \ell_Y \text{ unbounded from above}  & \Leftrightarrow & \exists \, i \in \{1,\hdots m\} \colon Y^{(i)} \in \langle Y^{(j)} : j \in \pa(i)  \rangle  \\ 
\text{(b)} & \text{MLE exists}  & \Leftrightarrow & \forall \, i \in \{1,\hdots m\} \colon Y^{(i)} \notin \langle  Y^{(j)} : j \in \pa(i)   \rangle  \\
\text{(c)} & \text{MLE exists uniquely} & \Leftrightarrow & \forall \, i \in \{1,\hdots m\} \colon Y^{(\pa(i) \cup i)} \text{ has full column rank}. \;\;\; \\ \end{matrix} \] 
\end{theorem}

The above theorem uses the convention that the linear hull of the empty set is the zero vector space. In particular, if a sample matrix $Y$ has a column of zeros, then $\ell_Y$ is unbounded from above, regardless of whether the corresponding vertex has parents in $\Gcal$.
For a DAG model on $\Gcal$ the maximum likelihood threshold is
\begin{equation}
    \label{eqn:mlt_dag}
    \mlt(\Gcal) : =  \max_{i 
\in \{1, \hdots, m\}} | \pa(i) | + 1,
\end{equation}
by Theorem~\ref{thm:DAG_Ymatrix}, 
see also ~\cite[Theorem 1]{drton2019maximum}.

\begin{remark}
There is a correspondence between the existence and uniqueness of the MLE and notions of stability from Geometric Invariant Theory, see~\cite[Theorem A.2]{makam2021symmetries} and~\cite[Theorem~10.6.4]{PhDthesisPhilipp}. For a DAG model, there are three equivalences:
 \begin{equation}
     \label{eqn:stability}
     \begin{matrix}
         Y \text{ unstable} &  \Leftrightarrow &  \text{ MLE does not exist} \\ 
     Y \text{ polystable} & \Leftrightarrow & \text{MLE exists} \\
      Y \text{ stable}  & \Leftrightarrow & \text{MLE exists uniquely.}
    \end{matrix} 
 \end{equation} 
Stability is under right multiplication by the set of invertible matrices $g$ with $\det g = 1$ and 
$g_{ij} = 0$ for all $i \neq j$ with $j \not\to i$ in $\Gcal$,
see~\cite[Definition A.1]{makam2021symmetries}.
This is a group if and only if the DAG $\Gcal$ is transitive, see~\cite[Proposition 5.1]{amendola2021invariant}. A DAG is transitive if it has the property that a path $k \to j \to i$ implies the presence of an edge $k \to i$.
\end{remark}

\section{Samples with non-unique MLE} \label{sec:classification}

The MLE does not exist given $Y$ in a directed Gaussian graphical model if certain sub-matrices of $Y$ have deficient column rank, as described in Section~\ref{sec:asprelim}. There are two ways this can happen. The first is that the number of samples $n$ is too small, the second is that the columns of $Y$ are not generic. We relate these two possibilities in Section~\ref{sec:relating}. This enables us to assume without loss of generality that $n \geq m$. 

With too few samples, the MLE will not exist, and with sufficiently many generic samples, the MLE will exist and be unique. Between these extremes, different possibilities occur, which we characterise in Section~\ref{sec:tdags}. Our result holds in the setting of transitive DAGs.

\subsection{Relating too few samples to non-generic samples}
\label{sec:relating}

We relate maximum likelihood estimation when $n \leq m$ to the setting $n \geq m$. 

\begin{proposition}
\label{prop:wlog}
    Fix sample data $Y \in \KK^{n \times m}$. 
Then the MLEs given $Y$ equal the MLEs given $Z$, where 
$Z \in \KK^{kn \times m}$ is the matrix obtained from $Y$ by duplicating it vertically $k$ times.  
\end{proposition}

\begin{proof}
The $\Lambda$-MLEs given $Z$ are  $\hat{\lambda}_{ij}$ that minimize each
$\| Z^{(i)} - \sum_{j \in \pa(i)} \lambda_{ij} Z^{(j)} \|^2$.
Since $\| Z^{(i)} - \sum_{j \in \pa(i)} \lambda_{ij} Z^{(j)} \|^2 = k \| Y^{(i)} - \sum_{j \in \pa(i)} \lambda_{ij} Y^{(j)}\|^2$, both norms are minimized at $\hat{\lambda}_{ij}$. Hence the $\Lambda$-MLEs given $Y$ and $Z$ agree. 
The $\Omega$-MLE components $\hat{\omega}_{i}$ given $Z$ are the 
residuals 
$\frac{1}{kn} \| Z^{(i)} - \sum_{j \in \pa(i)} \hat{\lambda}_{ij} Z^{(j)} \|^2$.
We have $\frac{1}{nk} \| Z^{(i)} \|^2 = \frac{k}{nk}  \| Y^{(i)} \|^2 = \frac{1}{n} \| Y^{(i)} \|^2$. The same norm computations hold for $Z^{(i)} - \sum_{j \in \pa(i)} \hat{\lambda}_{ij} Z^{(j)}$. Hence the $\Omega$-MLEs given $Y$ and $Z$ agree.
\end{proof}

Proposition~\ref{prop:wlog} allows us to assume without loss of generality that $n \geq m$. Indeed, if $n < m$ we let $k$ be minimal such that $kn \geq m$ and replace $Y$ by $Z \in \KK^{kn \times m}$.

\subsection{Possibilities for MLE existence and uniqueness}
\label{sec:tdags}

We study all possibilities that can arise for ML estimation in transitive DAG models.
The following theorem characterizes which MLE properties can occur.
An unshielded collider is an induced subgraph $i \to j \leftarrow k$ with no edge connecting $i$ and $k$. 
Recall from~\eqref{eqn:mlt_dag} that the maximum likelihood threshold $\mlt(\Gcal)$ of a DAG is $\max_{i \in \{1,\hdots,m\}} |\pa(i)| + 1$. 
The depth $d(\Gcal)$ of a DAG is the number of arrows in a longest path in $\Gcal$. 
If $\Gcal$ is transitive then $d(\Gcal) \leq \mlt(\Gcal) - 1$.

\begin{theorem}
Let $\Gcal$ be a transitive DAG and let $n$ denote the number of samples. The MLE properties that can occur in the DAG model on $\Gcal$ are as per Table~\ref{tab:tdag}.
\begin{table}[htbp]
\begin{tabular}{c|c|c|c}
 & does not exist & exists but not unique & unique \\
\hline 
$n \leq d(\Gcal)$ & \checkmark & & \\ 
\hline 
$d(\Gcal) < n < \mlt(\Gcal)$ & \checkmark & \checkmark & \\  
\hline 
$n \geq \mlt(\Gcal)$, unshielded colliders & \checkmark & \checkmark & \checkmark \\
\hline 
$n \geq \mlt(\Gcal)$, no unshielded colliders & \checkmark & & \checkmark 
\end{tabular}
\vspace{8pt}
\caption{Possible MLE properties for transitive DAG models}
\label{tab:tdag}
\end{table}
\end{theorem}

\begin{proof}
We use the characterisation of the existence and uniqueness of the MLE from Theorem~\ref{thm:DAG_Ymatrix}.
Define $d := d(\Gcal)$ and $\mlt: = \mlt(\Gcal)$.
By definition, there is a directed path
\begin{center}
		\begin{tikzcd}
			p_0 & p_1 \ar[l] & p_2 \ar[l] & \cdots \ar[l] & p_d \ar[l]
		\end{tikzcd}
\end{center}
in $\Gcal$. The transitivity of $\Gcal$ implies that $p_{j+1}, \ldots, p_d$ are parents of $p_j$ for all $j = 0,1,\ldots,d$.

Assume $n \leq d$. Then for any $Y \in \KK^{n \times m}$ the vectors $Y^{(p_j)} \in \KK^{n}$ for $j =0,1,\ldots,d$ are linearly dependent, since $n < d + 1$. Therefore, there is some non-trivial linear combination $\sum_j \lambda_j Y^{(p_j)} = 0$. Let $k$ be minimal such that $\lambda_k \neq 0$. Then $Y^{(p_k)}$ is a linear combination of (some of) its parent columns. Hence the MLE given $Y$ does not exist.

Next, assume $d < n < \mlt$. The MLE does not exist given almost all $Y$, by the definition of $\mlt$. 
However, the MLE does exist given  special samples, as follows.
Fix linear independent vectors $ f_0,  f_1, \ldots,  f_d \in \KK^{n}$ using $n \geq d+1$ and denote by $d(i)$ the number of arrows of a longest directed path in $\Gcal$ starting at $i$. Then $0 \leq d(i) \leq d$.  We have $d(i) = 0$ if and only if vertex $i$ is not in $\pa(j)$ for any $j$. Moreover, if $p \to i$ then $d(p) > d(i)$ by transitivity of $\Gcal$. Define $Y \in \KK^{n \times m}$ by setting $Y^{(i)} :=  f_{d(i)}$ for all $i \in \{1,\hdots,m\}$. The parent columns of $Y^{(i)} =  f_{d(i)}$ are all contained in $\{  f_{d(i) +1}, \ldots,  f_{d(\Gcal)}\}$, by construction. Thus $Y^{(i)}$ is not in the linear span of its parent columns and hence the MLE given $Y$ exists. 
Observe that there is a vertex $i$ in $\Gcal$ such that $n < 1 + |\pa(i)|$, since
$n < \mlt$. Therefore, for any $Y \in \KK^{n \times m}$ the submatrix $Y^{(i \cup \pa(i))}$ does not have full column rank, so the MLE given $Y$ is not unique.

Finally, assume $n \geq \mlt$. The MLE is unique given generic samples $Y \in \KK^{n \times m}$, by the definition of $\mlt$. 
The MLE does not exist for a matrix with a column of zeros, for example. 
It remains to see whether the MLE given $Y$ can exist but not be unique. 
If there is an unshielded collider $j \to i \leftarrow k$ in $\Gcal$, we create such a $Y$ by taking a generic $Y$ and replacing $Y^{(k)}$ by $Y^{(j)}$. 
Since $j \notin \pa(k)$ and $k \notin \pa(j)$, the MLE exists, but since two rows indexed by parents of $i$ are equal, it is not unique. 
We conclude with the case where there is no unshielded collider in $\Gcal$. Assume there is some sample matrix $Y$ such that the MLE is not unique given $Y$. By Theorem~\ref{thm:DAG_Ymatrix}(b) and~(c) there is some $i$ such that $Y^{(\pa(i))}$ does not have full column rank.
Let $J \subset \pa(i)$ denote the indexing set for those columns that appear with non-zero coefficient in a linear dependence relation among the columns of $Y^{(\pa(i))}$. Since there are no unshielded colliders in $\Gcal$, there is some $k \in J$ with $J \backslash \{ k \} \subset \pa(k)$. But then $Y^{(k)} \in \mathrm{span} \big\lbrace Y^{(j)} : j \in J \setminus \{k\} \big\rbrace \subseteq \mathrm{span} \big\lbrace Y^{(j)} : j \in \pa(k)  \big\rbrace$, which contradicts existence of the MLE.
\end{proof}

Section \ref{sec:relating} implies that we can always assume that we are in the situation where $n \geq m$, by duplicating samples enough times. So we may restrict our attention to the bottom two rows of Table~\ref{tab:tdag}. Given a sample $Y$ with non-unique MLE given $Y$, we will see in Section \ref{sec:fromcctoss} how to construct using a complete collineation a new sample $\widetilde{Y}$ with unique MLE given $\widetilde{Y}$. Then in Sections \ref{sec:MLEofsamplestab} and \ref{sec:uniqueMLEinlimit} we will relate the MLE given $\widetilde{Y}$ to the MLE(s) given $Y$, and show how $\widetilde{Y}$ can be used to resolve non-identifiability of the MLE given $Y$.

\section{From complete collineations to sample stabilisations} \label{sec:fromcctoss}

In this section we introduce the stabilisation of a sample. We call it a stabilisation because, as we will see, the MLE given any stabilisation of a sample is unique, see~\eqref{eqn:stability}. There are many ways we could obtain from a sample a new sample with unique MLE. The notion of stabilisation that we introduce here is based on complete collineations, and has the advantage that we can relate the MLE given a stabilisation to MLEs given the original sample, if they exist.  
We define sample stabilisations 
in Section \ref{subsec:samplestab}. We construct a parameter space for sample stabilisations as an algebraic variety in Section \ref{subsec:paramspace}. 

\begin{convention} \thlabel{nbiggerthanm} 
We assume $n \geq m$. This is without loss of generality, by Section \ref{sec:relating}.
\end{convention}

\subsection{Sample stabilisations from complete collineations}  \label{subsec:samplestab}

Defining sample stabilisations requires taking orthogonal complements in $\KK^n$ and $\KK^m$. To this end, we fix the standard inner products on $\KK^n$ and $\KK^m$, with $v \cdot w = \sum_{i=1}^n v_i w_i^{\ast}$ where $w_i^{\ast}$ denotes the complex conjugate.

\begin{definition}[Sample perturbations and stabilisations] \thlabel{sampleperturbandstab}
Fix a sample $f: \KK^m \to \KK^n$. A linear map $f': \KK^m \to \KK^n$ is a \emph{perturbation} of $f$, or \emph{$f$-perturbation}, if it satisfies the conditions: \begin{enumerate}[(i)]
\item $\im f'  \subseteq (\im f)^{\perp} $; \label{inclusion}
\item $(\ker f')^{\perp}  = \ker f$.  \label{equality}
\end{enumerate}
A \emph{stabilisation} of $f$, or \emph{$f$-stabilisation}, is a sum $\widetilde{f} 
 = f+ f'$, where $f'$ is an $f$-perturbation. 
\end{definition} 

Equivalently, a linear map $f': \KK^m \to \KK^n$ is an $f$-perturbation if and only if its rows and columns are orthogonal to the rows and columns of $f$, respectively, and $\dim \ker f' = \operatorname{dim} \im f$.

\begin{lemma}[\thref{firstmainresult} \ref{unique}] \thlabel{samplestabisstable}
Any stabilisation $\widetilde{f}$ of a sample $f$ has maximal rank. In particular, the MLE given $\widetilde{f}$ is unique in any DAG $\mathcal{G}$ on $m$ vertices. 
\end{lemma} 

\begin{proof} 
Write $\widetilde{f}$ as $f +f'$ where $f'$ is an $f$-perturbation. Since $n \geq m$, we wish to show that $\widetilde{f}$ has trivial kernel. To this end suppose that $\widetilde{f}(v) = 0$ for some $v \in \KK^m$. Write $v = v_1 + v_2$ where $v_1 \in \ker f$ and $v_2 \in (\ker f)^{\perp}$. Then $\widetilde{f}(v) = f'(v_1) + f(v_2)$. By \ref{inclusion} we know that $f'(v_1) \in (\im f)^{\perp}$ therefore $\widetilde{f}(v) = 0$ if and only if $f'(v_1) = f(v_2) = 0$. By \ref{equality} we have $v_1 \in (\ker f')^{\perp}$, therefore $v_1 = 0$. Since $v_2 \in (\ker f)^{\perp}$, we also have $v_2 = 0$. Therefore $v=0$ as required. Therefore $\widetilde{f}$ has maximal rank and so the MLE given $\widetilde{f}$ is unique in the DAG model on any DAG $\Gcal$ on $m$ vertices, by \thref{thm:DAG_Ymatrix}.
\end{proof} 

We now show how sample stabilisations can be constructed from complete collineations. 

\begin{construction}[An $f$-stabilisation from a complete collineation] \thlabel{construction1} Fix a sample $f$ and consider a complete collineation $([f_1],\hdots, [f_t])$ from $\PP(\KK^m)$ to $\PP(\KK^n)$ with $[f_1] = [f ]$. Choose an affine lift $(f_1,f_2, \hdots, f_t)$ with $f_1 = f $. Each $f_i$ is a non-zero map $ \ker f_{i-1} \to \coker  f_{i-1}$, with $f_t$ the first non-degenerate map (which must be injective since we are assuming $m \leq n$). 

We first explain how to turn each map $f_i$ into a map to $\KK^n$.
Using the standard inner product on $\KK^n$, we identify $\coker  f_1$ with $(\im  f_1)^{\perp}$, a subspace of $\KK^n$. In this way we view $f_2$ as a map $\ker  f_1 \to \KK^n$. 
The standard inner product on $\KK^n$ restricts to one on $(\im  f_1)^{\perp},$ which enables us to identify $\coker  f_2$ with the orthogonal complement of $\im  f_2$ inside $(\im  f_1)^{\perp}$:
$$ \coker f_2  = \{ x \in (\im f_1)^\perp \, \vert \, \langle x, y \rangle = 0 \text{ for all } y \in \im f_2\}. $$ 
Thus we can view $f_3$ as a map $\ker  f_2 \to \KK^n$. 
Proceeding in this way, we identify each $\coker  f_i$ as the orthogonal complement of $\im  f_i$ in $(\im  f_{i-1})^{\perp}$, and thus view $f_{i+1}$ as a map $\ker  f_{i} \to \KK^n$. Note that the images of each $f_i$ have pairwise trivial intersection. 

Next we explain how to turn each map $f_i$ into a map with domain $\KK^m$.
Let $$f_3': \ker  f_1 = \ker  f_2 \oplus (\ker  f_2)^{\perp}  \to \KK^n $$ 
denote the pre-composition of $f_3$ with the projection from $\ker  f_1$ to $\ker  f_2$. In the above equation, the orthogonal complement $(\ker  f_2)^{\perp}$ is taken inside $\ker  f_1$. 
Let $$f_{i+1}': \ker  f_1 \to \KK^n$$ denote the pre-composition of $f_{i+1}$ with the sequence of projections $\ker  f_1  \twoheadrightarrow  \cdots \twoheadrightarrow \ker  f_i$. The process ends when we reach $f_t': \ker  f_1 \to \KK^n$, whose restriction to $\ker f_{t-1}$ has trivial kernel.   
Since the images of each $f_i'$ have pairwise trivial intersection, we obtain an injective map $$f_2' + \cdots + f_t': \ker f_1 \to \KK^n$$ with image contained in $(\im f_1)^{\perp}$. Pre-composing with the projection $\KK^m = \ker f_1 \oplus (\ker f_1)^{\perp} \to \ker f_1$ gives a map $f': \KK^m \to \KK^n$ with kernel $(\operatorname{ker} f_1)^{\perp}$ and image contained in $(\im f_1)^{\perp}$. 

\begin{lemma} \thlabel{fperturb}
The map $f': \KK^m \to \KK^n$ is an $f$-perturbation. 
\end{lemma}

\begin{proof}
By construction we have $( \ker f')^{\perp} = \ker f$ and $\im f' \subseteq (\im f)^{\perp}$. Hence both conditions of \thref{sampleperturbandstab} required for $f'$ to be an $f$-perturbation are satisfied. 
\end{proof}

By \thref{fperturb}, we set $\widetilde{f} : = f+ f'$ to obtain an $f$-stabilisation. 
\end{construction} 

There are no choices involved in this construction, beyond the standard bases and inner products on $\KK^n$ and $\KK^m$, which are set once and for all. Thus we have a canonical way of obtaining an $f$-stabilisation given a sample $f$ and an affine lift of a complete collineation with first term $[f]$. 

\begin{proposition}  \thlabel{construction}
Given a sample $f$, an affine lift $(f_1, \hdots, f_t)$ of a complete collineation from $\PP(\KK^m)$ to $\PP(\KK^n)$ with first term $f_1=f$ uniquely determines an $f$-perturbation $f'$ and an $f$-stabilisation $\widetilde{f} = f + f'$, via \thref{construction1}. 
\end{proposition}

\begin{example}[Illustration of \thref{construction1}]
Let $$ f = f_1 = \begin{pmatrix} 1 & 0 & 0 \\
0 & 1 & 0 \\
0 & 0 & 0 \\
0 & 0 & 0 
\end{pmatrix}$$ so that $m=3$ and $n=4$. Let $\{b_1,b_2,b_3\}$ and $\{e_1,e_2,e_3,e_4\}$ denote the standard bases for $\KK^3$ and $\KK^4$ respectively. Then $\ker f = \langle b_3 \rangle$ while $(\im f)^{\perp} = \langle e_3, e_4 \rangle$. A non-zero map $f_2: \ker f \to (\im f)^{\perp}$ is of the form $b_3 \mapsto c_1 e_3 + c_2 e_4$ for some $c_1,c_2 \in \KK$ not both zero. This map is necessarily injective, so $(f_1,f_2)$ is a an affine lift of a complete collineation from $\PP(\KK^3)$ to $\PP(\KK^4)$. Then $$ 
\widetilde{f} = f + f' = \begin{pmatrix} 1 & 0 & 0 \\
0 & 1 & 0 \\
0 & 0 & c_1 \\
0 & 0 & c_2 
\end{pmatrix} . $$
\end{example}

\subsection{The parameter space of sample stabilisations}   \label{subsec:paramspace}

A perturbation of a sample $f$ is a linear map from $\KK^m$ to $\KK^n$, or alternatively an element in $X: = \KK^{n \times m}$. We describe the subvariety $X_f$ of $X$ parametrising $f$-perturbations. This is a parameter space for $f$-stabilisations. 

Fix a sample $f$ and let $r: = \operatorname{dim} \im  f$. Let $Y_f \subseteq X$ be the subspace of maps $f': \KK^m \to \KK^n$ that descend to a map $$\KK^m / (\ker  f)^{\perp} \cong \ker  f   \to (  \im  f)^{\perp}.$$ In other words, $f' \in Y_f$ if and only if the columns of $f'$ are orthogonal to the columns of $f$ and the rows of $f'$ are orthogonal to the rows of $f$. 
The space $Y_f$ is cut out by linear equations in~$X$.
Let $X_{m-r}$ be the rank $m-r$ matrices in $X$. This is a locally closed subvariety of $X$, as it is closed inside the open subvariety given by matrices of rank less than or equal to $m-r$.  Set \begin{equation} X_f : =  X_{m-r} \cap Y_f. \label{defofpsYf} \end{equation} 

 \begin{proposition}[\thref{secondmainresult}\ref{psforstab}] \thlabel{psfperturb}
Fix $f \in \KK^{n \times m}$. Then $f' \in \KK^{n \times m }$ is an $f$-perturbation if and only if $f' \in X_f$.     
 \end{proposition}

 \begin{proof} 
If $f'$ is an $f$-perturbation then it descends to a map from $\KK^m / (\ker  f)^{\perp} \cong \ker f$ to $(  \im  f)^{\perp}$, by definition. Thus it lies in $Y_f$. If $f' \in Y_f$, then it is an $f$-perturbation if and only if $\dim \ker f' = \dim \im f = r$, or equivalently if and only if $\dim \im f' = m -r$. Hence $f' \in Y_f$ lies in $X_f$ if and only if it has rank $m-r$. 
 \end{proof} 

 \begin{definition}[Parameter space of $f$-stabilisations] \thlabel{def:ps}
Given a sample $f$, the subvariety $X_f \subseteq X = \KK^{n \times m}$ defined in \eqref{defofpsYf} is the \emph{parameter space of $f$-stabilisations}.      
 \end{definition}

\begin{remark}[Link between $X_f$ and the moduli space of complete collineations]
The moduli space $\mathcal{M}$ of complete collineations from $\PP(\KK^m)$ to $\PP(\KK^n)$ can be constructed as a blow-up of $\PP(\Hom(\KK^m,\KK^n))$, see
Section \ref{subsec:pointsofms}. Hence there is a surjective morphism $$\pi: \mathcal{M} \to \PP(\Hom(\KK^m,\KK^n)),$$ 
which maps $([f_1],\hdots,[f_t])$ to $[f_1 ]$. Let $\mathcal{M}_{[f]} : = \pi^{-1}([f])$ and let $\mathcal{M}^{\operatorname{aff}}_{f}$ denote the space over $\mathcal{M}_{[f]}$ with fibre over each point  $([f_1],[f_2],\hdots,[f_t])$ given by $t-1$ copies of $\KK^{\ast} : = \KK \setminus \{0\}$, parametrising a choice of non-zero affine lift with first term $f$.
\thref{construction} then gives a map from $\mathcal{M}^{\operatorname{aff}}_{f}$ to the parameter space $X_f$ of $f$-stabilisations. 
\end{remark}

\section{MLEs given stabilisations}  \label{sec:MLEofsamplestab}

Let $\mathcal{G}$ be a connected DAG on $m$ vertices. We study the MLE given a stabilisation $\widetilde{f} : = f+ f'$ in the DAG model on $\Gcal$. We obtain necessary and sufficient conditions for the MLE given an $f$-stabilisation $\widetilde{f}$ to be an MLE given $f$, in Section \ref{subsec:uniqueMLE}.
If an MLE given $f$ does not exist, we study the analogous question for the $\Lambda$-MLE, which always exists by Proposition~\ref{prop:lambda_and_omega}.
We study which MLEs can be obtained as the MLE given a stabilisation in Section \ref{subsec:whichMLEs}.

\subsection{When is the MLE given an $f$-stabilisation an MLE given $f$?}
\label{subsec:uniqueMLE}

\begin{proposition}[When is the $\Lambda$-MLE given an $f$-stabilisation a $\Lambda$-MLE given $f$?]  \thlabel{MLEofstab}
Fix a DAG $\mathcal{G}$, a sample $f$, and an $f$-stabilisation $\widetilde{f}= f + f'$. Let $f_i$ be the columns of $f$ and $v_i$ the columns of $f'$. 
The $\Lambda$-MLE 
given $\widetilde{f}$ is a $\Lambda$-MLE given $f$ in the DAG model on $\Gcal$ if and only if  \begin{equation} \overline{f_i} +  \overline{v_i} \in \langle f_j + v_j : j \to i \rangle, \label{condition} \end{equation} 
for all child vertices $i$,
where $\overline{f_i} := \pi_{\langle f_j   :  j \to i \rangle}(f_i)$ and $\overline{v_i} := \pi_{\langle v_j  :  j \to i \rangle}(v_i)$. 
\end{proposition} 

\begin{proof}
For a vertex $i$ of $\mathcal{G}$, we call the components $\lambda_{ij}$ of the $\Lambda$-MLE indexed by arrows $j \to i$ the $\Lambda_i$-MLE. 
We show that the $\Lambda_i$-MLE given $\widetilde{f}$ is a $\Lambda_i$-MLE given $f$ if and only if \eqref{condition} holds.
The $\Lambda_i$-MLE for $\widetilde{f}$ consists of coefficients $\{\lambda_{ij}\}_{j \to i}$ such that \begin{equation} \pi_{\langle f_j + v_j  :  j \to i \rangle } (f_i + v_i) = \sum_{j \to i} \lambda_{ij}(f_j + v_j).
\label{MLEdef} \end{equation}
Using the containment $ \langle f_j + v_j : j \to i \rangle \subseteq \langle f_j : j \to i \rangle \oplus \langle v_j : j \to i \rangle$, we obtain 
\begin{align*} 
\pi_{\langle f_j + v_j : j \to i \rangle}  (f_i + v_i) & = \pi_{\langle f_j + v_j : j \to i \rangle}  (\pi_{\langle f_j  : j \to i \rangle \oplus \langle v_j :  j \to i \rangle} (f_i + v_i))\\
 & = \pi_{\langle f_j + v_j : j \to i \rangle}  (\pi_{\langle f_j  : j \to i \rangle \oplus \langle v_j  : j \to i \rangle} (f_i )) +  \pi_{\langle f_j + v_j : j \to i \rangle}  (\pi_{\langle f_j  : j \to i \rangle \oplus \langle v_j  : j \to i \rangle} (v_i)) \\
& = \pi_{\langle f_j + v_j : j \to i \rangle} ( \overline{f_i} + \overline{v_i}),
\end{align*}
since $\langle v_1,\hdots, v_m \rangle$ and $\langle f_1, \hdots, f_m \rangle$ are orthogonal. 
If $\overline{f_i} + \overline{v_i} \in 
\langle  f_j + v_j : j \to i \rangle$, then
$$ \pi_{\langle f_j + v_j : j \to i \rangle}  (f_i + v_i) = \overline{f_i} + \overline{v_i}.$$
This is $\sum_{j \to i} \lambda_{ij} f_j + \sum_{j \to i} \lambda_{ij} v_j$, by \eqref{MLEdef}.
We conclude that 
$\overline{f_i} = \sum_{j \to i} \lambda_{ij} f_j$, 
again by the orthogonality of $\langle f_1, \hdots, f_m \rangle$ and $\langle v_1,\hdots, v_m \rangle$.
Hence the $\Lambda_i$-MLE for $\widetilde{f}$ is a $\Lambda_i$-MLE for $f$. 

Conversely, assume that the $\Lambda_i$-MLE for $\widetilde{f}$ is a $\Lambda_i$-MLE for $f$.
This means $\overline{f}_i = \sum_{j \to i} \lambda_{ij} f_j$ for the same cofficients $\lambda_{ij}$ as in~\eqref{MLEdef}.
Define $\overline{v}_i = \sum_{j \to i} \nu_{ij} v_j$ and $x = \sum_{j \to i} (\nu_{ij} - \lambda_{ij}) v_j$.
Then 
$$ \pi_{\langle  f_j + v_j :  j \to i \rangle}( \overline{f_i} + \overline{v_i}) =  \sum_{j \to i} \lambda_{ij}(f_j + v_j) = \sum_{j \to i} \lambda_{ij} f_j + \sum_{j \to i} \lambda_{ij} v_j = \overline{f_i} + \overline{v_i} - x. $$
We have $x \in \langle v_j : j \to i \rangle$, by definition. 
Moreover, $x \in \langle  f_j + v_j : j \to j \rangle^\perp$, 
since the projection of $\overline{f_i} + \overline{v_i}$ onto $\langle  f_j + v_j : j \to i \rangle$ differs from $\overline{f_i} + \overline{v_i}$ by $x$.
Hence $ x \cdot( f_j + v_j) = 0$ for all $j \to i$. But $ x \cdot ( f_j + v_j) =  x \cdot v_j$, since $f_j$ and $x$ are orthogonal for all $j \to i$. Hence $x \in \langle v_j : j \to i \rangle^\perp \cap \langle v_j : j \to i \rangle$ and we conclude that $x=0$.
\end{proof}

If an MLE given $f$ exists, we can ask when the MLE given an $f$-stabilisation $\widetilde{f}$ is an MLE given $f$. \thref{OmegaMLEofstab} below gives a complete answer -- this is \thref{firstmainresult}\ref{iff}.

\begin{corollary}[When is the MLE given an $f$-stabilisation an MLE given $f$?] \thlabel{OmegaMLEofstab}
Fix a DAG $\mathcal{G}$, sample $f$, and $f$-stabilisation $\widetilde{f} = f+ f'$. Let $f_i$ denote the columns of $f$ and $v_i$ the columns of $f'$. Then 
the MLE given $\widetilde{f}$ in the DAG model on $\Gcal$ is an MLE given $f$ if and only if 
 \begin{equation} v_i \in \langle v_j : j \to i \rangle \text{ and }  \overline{f_i} +  v_i \in \langle f_j + v_j : j \to i \rangle  \label{omega}
 \end{equation}
 for all child vertices $i$ of $\Gcal$, where $\overline{f_i} := \pi_{\langle f_j   :  j \to i \rangle}(f_i)$. 
\end{corollary}

\begin{remark}[Sanity check]
It follows from \thref{OmegaMLEofstab} that if the MLE given $f$ does not exist, then the condition given in \eqref{omega} can never be satisfied by an $f$-stabilisation $\widetilde{f}$. This can be seen directly as follows. Suppose that the MLE given $f$ does not exist. Then there is some $i$ with $f_i \in \langle f_j : j \to i \rangle$, so that $\overline{f_i} = f_i$.  Suppose $\widetilde{f}$ is an $f$-stabilisation satisfying \eqref{omega}. Then $f_i  + v_i \in \langle f_j + v_j : j \to i \rangle$, contradicting the existence of the MLE given $\widetilde{f}$, by \thref{thm:DAG_Ymatrix}. 
\end{remark}

\begin{proof}[Proof of \thref{OmegaMLEofstab}]
Define the $\Lambda_i$-MLE as in the proof of \thref{MLEofstab} and similarly define the $\Omega_i$-MLE to be the component $\omega_i$ of the $\Omega$-MLE indexed by $i$. 
Suppose that the $\Lambda_i$-MLE and $\Omega_i$-MLE given $\widetilde{f}$ are a $\Lambda_i$-MLE and $\Omega_i$-MLE given $f$. Then $\overline{f_i}+ \overline{v_i} \in \langle f_j + v_j :j \to i \rangle$, 
by \thref{MLEofstab}.
The $\Omega_i$-MLE given $\widetilde{f}$ is the norm of $\pi_{\langle f_j + v_j :j\to i \rangle} (f_i + v_i) - f_i - v_i$.  
This is, equivalently, the norm of $\overline{f_i} - f_i + \overline{v_i} - v_i$
since $$\pi_{\langle f_j + v_j :j \to i \rangle} (f_i + v_i)  = \pi_{\langle f_j + v_j :j \to i \rangle} (\overline{f_i} + \overline{v_i}) = \overline{f_i} + \overline{v_i}.$$   
The $\Omega_i$-MLE given $f$ is the norm of $\overline{f_i} - f_i.$ For the MLEs to coincide, the vectors  $\overline{f_i} - f_i + \overline{v_i} - v_i$ and $\overline{f_i} - f_i$ must have the same norm. But given that $\overline{v_i} - v_i$ lies in $\langle \overline{f_i} - f_i \rangle^\perp$, this means the vectors must be equal, so that $v_i = \overline{v_i}$. Hence $v_i \in \langle v_j :j \to i \rangle$. 

For the other direction, suppose that $\overline{f_i}+ \overline{v_i} \in \langle f_j + v_j :j \to i \rangle$ and $v_i \in \langle v_j :j \to i \rangle.$ The first condition ensures that the $\Lambda_i$-MLE given $\widetilde{f}$ is a $\Lambda_i$-MLE given $f$, by \thref{MLEofstab}. It remains to show that the $\Omega_i$-MLE given $\widetilde{f}$ is the $\Omega_i$-MLE given $f$. But this follows from the fact that $v_i \in \langle v_j : j \to i \rangle$, by the same calculations as in the previous paragraph.    
\end{proof}

\begin{remark}\thlabel{examplestocome} 
It is reasonable to wonder whether there always exists a stabilisation of $f$ whose MLE is an MLE given $f$. \thref{unstableexample2} will show that this is not necessarily the case. 
\end{remark} 

\subsection{When is an MLE given $f$ the MLE given an $f$-stabilisation?} \label{subsec:whichMLEs}

\thref{OmegaMLEofstab} gives necessary and sufficient conditions for the MLE given an $f$-stabilisation to be an MLE given $f$. It is natural to ask which MLEs given $f$ are the MLE given some $f$-stabilisation. We reformulate this question geometrically, showing that it reduces to asking whether a locally closed subvariety of the parameter space $X_f$ from \thref{def:ps} is non-empty. 
As a first step, we characterise when, for a fixed MLE $\alpha$ given $f$, the MLE given an $f$-stabilisation is also $\alpha$. 

\begin{proposition} \thlabel{firststep} 
Let $\alpha$ be an MLE given $f$ in a DAG model on $\Gcal$. Let $\lambda$ denote the $\Lambda$-MLE part of $\alpha$. Fix an $f$-perturbation $f'$ and let $v_i$ denote its columns. Then the MLE given $\widetilde{f} : = f+ f'$ is $\alpha$ if and only if, for every child vertex $i$, \begin{equation}
v_i = \sum_{j \to i} \lambda_{ij} v_j.
\end{equation} 
\end{proposition} 

\begin{proof}
Suppose that $\alpha$ is the MLE given $\widetilde{f}$. Since $\alpha$ is also an MLE given $f$, we have \begin{equation*} v_i \in \langle v_j :j \to i \rangle \text{ and }  \overline{f_i} +  v_i \in \langle f_j + v_j :j \to i \rangle
 \end{equation*} for all child vertices $i$, where $\overline{f_i} := \pi_{\langle f_j : j \to i \rangle}(f_i)$, by \thref{OmegaMLEofstab}.
Since  $\lambda$ is a $\Lambda$-MLE given $f$, we know that $\overline{f_i} = \sum_{j \to i} \lambda_{ij} f_j$. Since $\lambda$ is also an $\Lambda$-MLE given $\widetilde{f}$, we have $$ \pi_{\langle f_j + v_j :j \to i \rangle}( f_i + v_i) =  \sum_{j \to i} \lambda_{ij} (f_j +  v_j) = \pi_{\langle f_j + v_j :j \to i \rangle} (\overline{f_i} + \overline{v_i} ) =  \overline{f_i} + \overline{v_i} =  \overline{f_i} + v_i .$$ It follows from orthogonality of the $f_i$ and $v_i$ that $v_i = \sum_{j \to i} \lambda_{ij} v_j$, as required.

 Conversely, suppose that $v_i = \sum_{j \to i} \lambda_{ij} v_j$ for all child vertices $i$. Then $\overline{v_i} = v_i$ and, since $\overline{f_i} = \sum_{j \to i} \lambda_{ij} f_j$, it follows that $\overline{f_i} + \overline{v_i } = \sum_{j \to i} \lambda_{ij} (f_j + v_j) \in \langle f_j + v_j :j \to i \rangle$. 
 Hence
 $$ \pi_{\langle f_j + v_j : j \to i \rangle} ( f_i + v_i) = \pi_{\langle f_j + v_j : j \to i \rangle} ( \overline{f_i}+  \overline{v_i}) = \overline{f_i} + \overline{v_i}   = \sum_{j \to i} \lambda_{ij} (f_j + v_j), $$ so that $\lambda$ is the $\Lambda$-MLE given $\widetilde{f}$.
 Since $\overline{v_i} = v_i$ and $ \overline{f_i} + \overline{v_i } = \sum_{j \to i} \lambda_{ij} (f_j + v_j) \in \langle f_j + v_j :j \to i \rangle$ for all $i$, by \thref{OmegaMLEofstab} we know that the $\Omega$-MLE given $\widetilde{f}$ is also an $\Omega$-MLE given $f$. The latter is unique. Therefore $\alpha$ is the MLE given $\widetilde{f}$.  
\end{proof}

We can use \thref{firststep} to characterise geometrically when an MLE given $f$ is the MLE given an $f$-stabilisation. Fix $\alpha$ an MLE given $f$, with $\lambda$ its $\Lambda$-MLE  component. For every vertex $i$ let $Y_{\alpha,i} \subseteq X = \KK^{n \times m}$ denote the linear subspace defined by 
\begin{equation} v_i - \sum_{j \to i} \lambda_{ij} v_j = 0. \label{explicit} 
\end{equation}  
Define $Y_{\alpha} = \bigcap_{i} Y_{\alpha,i} \subseteq X$ and \begin{equation} X_{f, \alpha} = Y_{\alpha} \cap X_f. \end{equation} By construction, $X_{f,\alpha}$ is a closed subvariety of $X_f$. Moreover, by \thref{firststep} we have that $\widetilde{f} \in X_{f, \alpha}$ if and only if the MLE given $\widetilde{f}$ is $\alpha$. We obtain the following, which is \thref{secondmainresult}\ref{ps2}. 

\begin{corollary}[When is an MLE given $f$ the MLE given an $f$-stabilisation?]
 \thlabel{MLEfromMLEofstab} Let $\alpha$ be an MLE given a sample $f$. Then $\alpha$ is the MLE given an $f$-stabilisation if and only if $X_{f, \alpha} \neq \emptyset$. 
\end{corollary}

\begin{definition}[Parameter space of $f$-stabilisations with MLE $\alpha$] \thlabel{defofps}
Let $f$ denote a sample and $\alpha$ an MLE given $f$. Then the closed subvariety $X_{f, \alpha} \subseteq X_f$ 
is the \emph{parameter space of $f$-stabilisations $\widetilde{f}$ such that $\alpha$ is the MLE given $\widetilde{f}$}. 
\end{definition}

The question of which MLEs given $f$ can be obtained as MLEs given an $f$-stabilisation amounts therefore to determining whether $X_{f,\alpha}$ is non-empty. We have given defining equations for $X_{f,\alpha}$ in \eqref{explicit}. This means we can apply techniques from algebraic geometry to determine whether the subvariety $X_{f, \alpha}$ is non-empty -- see \cite{Giusti1994} for $\KK=\CC$, and \cite{basu2006algorithms} for $\KK= \RR$.
We describe $X_{f,\alpha}$ explicitly for linear regression models in Section \ref{subsec:bounded}.

\section{Solutions of underdetermined linear systems using stabilisations} \label{sec:limitsolutions}

We investigate how solutions to underdetermined linear systems of a particular form can be obtained from limits of solutions to related full rank systems. As in Section \ref{sec:MLEofsamplestab} we fix the standard inner product on $\KK^n$.

Fix a matrix $A \in \KK^{n \times p}$ and a vector $b \in \KK^n$, with $n \geq p$. 
Let $\pi_A(b)$ denote the projection of $b$ onto the column space of $A$. We consider linear systems of the form $$ Ax = \pi_{A} (b).$$  
One solution is given by the pseudo-inverse
$x = A^+ \pi_A (b)$.
It is the unique solution
if and only if the matrix $A$ has full column rank, in which case it is $x = (A\T A)^{-1} A\T \pi_A(b)$.
 Note that $A^{\dagger} \pi_A(b) = A^{\dagger} b$, by properties of the pseudoinverse.
 
In this section we give two proofs of the following result.

\begin{theorem}
    \thlabel{thm:linear_system}
Let $A(\epsilon) = A + \epsilon E$ and $b(\epsilon) = b + \epsilon v$ where $A(\epsilon)$ has full column rank for each $\epsilon \neq 0$ and the columns of $A$ and $b$ are orthogonal to the columns of $E$ and to $v$. Let $x(\epsilon) = A(\epsilon)^+ \pi_{A (\epsilon)} (b(\epsilon))$. Then the limit  
$$ x: = \lim_{\epsilon \to 0} x(\epsilon)$$
exists  
and it is a solution to $Ax = \pi_A(b)$.
\end{theorem}

We give a formula for $x:= \lim_{\epsilon \to 0} x(\epsilon)$ in \thref{cor:limit_formula}. In Section~\ref{sec:uniqueMLEinlimit} we will think of this limit solution as a  way to choose a unique $\Lambda$-MLE from a choice of infinitely many. An alternative choice of solution to $Ax = \pi_A (b)$ is the solution $x = A^+ \pi_A(b)$, which is the minimal norm solution. We will see in Section \ref{sec:linearregression} that in the special case of linear regression models, the limit solution agrees with the minimal norm solution.

The challenge in proving \thref{thm:linear_system} is that the pseudo-inverse is not necessarily a continuous function in the elements of the matrix. It is continuous if and only if $A(\epsilon)$ and $A$ have the same rank for sufficiently small $\epsilon$, see \cite{Stewart1969}. When they do not have the same rank, the limit $\operatorname{lim}_{\epsilon \to 0}(A(\epsilon))^{+}$ does not exist.   
Luckily, we are not interested in $A(\epsilon)^{+}$ and its limit but rather in $A(\epsilon)^{+} \pi_{A(\epsilon)}(b(\epsilon))$ and its limit. As we will see, multiplying by $\pi_{A(\epsilon)}(b(\epsilon))$ resolves the discontinuity to give a well-defined limit solution.

Theorem~\ref{thm:linear_system} has the following geometric interpretation. The vector $x(\epsilon) =  A(\epsilon)^{+} \pi_{A(\epsilon)}( b(\epsilon))$ gives the coefficients for the projection of $b(\epsilon)$ onto the column space of $A(\epsilon)$: $$ \pi_{A(\epsilon)}(b(\epsilon)) = \sum_{i=1}^p x_i(\epsilon) A(\epsilon)_i$$ where $x_i(\epsilon)$ denotes the $i$-th entry of $x$ and $A(\epsilon)_i$ the $i$-th column of $A(\epsilon)$. \thref{thm:linear_system} implies that these coefficients do not go off to infinity. Now we also have \begin{equation} \pi_{A(\epsilon)}(b(\epsilon)) =\pi_{A(\epsilon)} (\overline{b} + \epsilon \overline{v}), \label{keyequality}
\end{equation} where $\overline{b} = \pi_{A}(b)$ and $\overline{v} = \pi_E(v)$. This follows from the proof of \thref{MLEofstab}, since the columns of $A$ and $b$ are orthogonal to the columns of $E$ and to $v$. So instead of projecting $b(\epsilon)$ we can project $\overline{b} + \epsilon \overline{v}$, which is near the column space of $A$ for small $\epsilon$, and hence also near the column space of $A(\epsilon)$ for small $\epsilon$. Therefore \thref{thm:linear_system} says, roughly, that if we project a vector onto a subspace that is `very close' to it, the coefficients don't go off to infinity. This assumption is important because if we are projecting a vector that is `far away' from our subspace, the limit may not exist. We give two examples below to illustrate the two behaviours.

\begin{example}
\thlabel{ex:works} 
    Fix 
    $$ A(\epsilon) = \begin{pmatrix} 1 & 0 \\ 0 & 0 \end{pmatrix} +  \epsilon \begin{pmatrix} 0 & 0 \\ 0 & 1 \end{pmatrix}, \quad b(\epsilon) = \begin{pmatrix} 0 \\ 0 \end{pmatrix} + \epsilon \begin{pmatrix} 0 \\ 1\end{pmatrix}.$$ Then the conditions of \thref{thm:linear_system} are satisfied for $A(\epsilon)$ and $b(\epsilon)$ so $x(\epsilon)$ has a limit as $\epsilon \to 0$. It can be calculated as follows. 
We have $\pi_A(b) = 0$, so the system $Ax = \pi_{A}(b)$ has solutions $ce_2$ for any $c \in \KK$, where $e_2 = \begin{pmatrix} 0 & 1 \end{pmatrix}\T$. Moreover, we have $\pi_{A(\epsilon)}(b(\epsilon)) = b(\epsilon)$, so $A(\epsilon) x = \pi_{A(\epsilon)} (b(\epsilon))$ has unique solution $x(\epsilon) = e_2$.
    Thus $x(\epsilon)$ has a limit as $\epsilon \to 0$, and this is a solution to $A x = \pi_A(b)$. Note that the limit $x = e_2$ is not the solution obtained from the pseudo-inverse $A^+ b$, which is 
    the minimal norm solution $\begin{pmatrix} 0 & 0 \end{pmatrix}\T.$
\end{example}

\begin{example}
\thlabel{ex:fails}
Fix 
    $$ A(\epsilon) = A + \epsilon E := \begin{pmatrix} 1 & 0 \\ 0 & 0 \end{pmatrix} + \epsilon \begin{pmatrix} 0 & 0 \\ 0 & 1 \end{pmatrix}, \quad b(\epsilon) = \begin{pmatrix} 0 \\ 1 \end{pmatrix} + \epsilon \begin{pmatrix} 0 \\ 0 \end{pmatrix} .$$ Since $\begin{pmatrix} 0 & 1 \end{pmatrix}\T $ is not orthogonal to the second column $E$, the conditions of \thref{thm:linear_system} are not satisfied. We show that in this case $x(\epsilon)$ does not have a finite limit as $\epsilon \to 0$. 
Since $\pi_{\langle A \rangle}(b) = 0$, the system $Ax = \pi_{\langle A \rangle}(b)$ has solutions $ce_2$ for any $c \in \KK$.
Since $\pi_{A(\epsilon)}(b(\epsilon)) = b(\epsilon)$, the system $A(\epsilon) x = \pi_{A(\epsilon)} b(\epsilon)$ has unique solution $x(\epsilon) = \frac{1}{\epsilon} e_2$. This does not have a finite limit as $\epsilon \to 0$.
\end{example}

We prove the second part of \thref{thm:linear_system}, namely that whenever it exists the limit $x = \lim_{\epsilon \to 0} x(\epsilon)$ is a solution to $Ax = \pi_A(b)$, in \thref{prop:if_limit} below. 

\begin{proposition}[The limit is a solution]
\thlabel{prop:if_limit}
Let $A(\epsilon) = A+ \epsilon E$ and $b(\epsilon) = b + \epsilon v$ be as in \thref{thm:linear_system}. Let $x(\epsilon)$ denote the unique solution to $ A(\epsilon) x(\epsilon) = \pi_{A(\epsilon)}(b(\epsilon))$ for $\epsilon \neq 0$.  
Assume that $x:= \lim_{\epsilon \to 0} x(\epsilon)$ exists. Then $x$ is a solution to $Ax = \pi_A(b)$. 
\end{proposition} 

\begin{proof}
We show that $\pi_{A(\epsilon)}(b(\epsilon))$ tends to $\pi_A (b)$ as $\epsilon \to 0$. Let $f_i$ and $v_i$ denote the columns of $A$ and $E$ respectively. 
For $\epsilon > 0$, each $\langle f_i + \epsilon v_i : 1 \leq i \leq p \rangle$       determines a point $L_{\epsilon}$ in the Grassmannian $G(p,n)$ of $p$-dimensional subspaces of $\KK^n$. 
     Since $G(p,n)$ is compact, there is a limit subspace $L_0 \in G(p,n)$ as $\epsilon \to 0$. Choose a basis $b_1^0, \hdots, b_p^0$ for $L_0$. By the geometric version of Nakayama's lemma, this basis can be lifted to a basis $b_1^{\epsilon}, \hdots, b_p^{\epsilon}$ of $L_{\epsilon}$ for small $\epsilon$. 
     
     Let $M_{\epsilon}$ be the $n \times p$ matrix with columns $b_1^{\epsilon}, \hdots, b_p^{\epsilon}$. Then $$\pi_{A(\epsilon)} (b(\epsilon)) = \pi_{A(\epsilon)}(\overline{b} + \epsilon \overline{v}) =  \pi_{L_{\epsilon}}(\overline{b} + \epsilon \overline{v}) = M_{\epsilon} M_{\epsilon}^{+} (\overline{b} + \epsilon \overline{v}).$$ Recall that the second equality follows from orthogonality of the columns of $A$ and $b$ with the columns of $E$ and $v$. 
     Now $\operatorname{lim}_{\epsilon \to 0} M_{\epsilon} = M_0$, and $M_0$ has the same rank as $M_{\epsilon}$ for $\epsilon \neq 0$. Therefore we have $\operatorname{lim}_{\epsilon \to 0} M_{\epsilon}^{+}  = M_0^{+}$, by~\cite{Israel1966} (see also \cite{Stewart1969}). Therefore \begin{align*} 
      \operatorname{lim}_{\epsilon \to 0} \pi_{A(\epsilon)} (b(\epsilon)) &  =  \operatorname{lim}_{\epsilon \to 0} M_{\epsilon} M_{\epsilon}^{+} (\overline{b} + \epsilon \overline{v}) = M_0 M_0^{+} (\overline{b}) \\
      & = \pi_{L_0}(\overline{b})  = \pi_A (\overline{b}) = \overline{b} = \pi_A(b). \qedhere \end{align*}  
\end{proof}

\begin{remark}[\thref{ex:works,ex:fails} revisited]
The choices involved in the proof of \thref{prop:if_limit} can be made explicit if we work with \thref{ex:works}. In this example we have $n,p=2$, so $G(p,n)$ consists of a single point, namely $\KK^2$. Therefore $L_{\epsilon} = L_0 = \KK^2$ for each $\epsilon \neq 0$. We can take the standard basis $\{b_1^0 = e_1,b_2^0= e^2\}$ for $L_0$. This same basis is a lift to a basis of $L_{\epsilon} = \KK^2$ for any $\epsilon$, i.e.\ we take $b_1^{\epsilon} = e_1$ and $b_2^{\epsilon}=e_2$. Then $M_{\epsilon}$ is the two by two identity matrix. 

\thref{ex:fails} shows that the condition that the columns of $A$ and $b$ are orthogonal to the columns of $E$ and $v$ in \thref{prop:if_limit} is necessary. The proof of \thref{prop:if_limit} fails for this example because it is not the case that $b$ is orthogonal to the columns of $E$, yet this condition is needed to ensure that $\pi_{A(\epsilon)}(b(\epsilon)) = \pi_{A(\epsilon)}(\overline{b} + \epsilon \overline{v}).$ In \thref{ex:fails} the left-hand side is the vector $\begin{pmatrix} 0 & \epsilon \end{pmatrix}\T $ while the right-hand side is the zero vector. 
\end{remark}

\subsection{Geometric proof}  We prove the first part of \thref{thm:linear_system}, namely that the limit $\operatorname{lim}_{\epsilon \to 0} x(\epsilon)$ exists, via a geometric argument. Recall that $x(\epsilon) $ for $\epsilon \neq 0$ is defined by the equation $$ A(\epsilon) x(\epsilon) = \pi_{A(\epsilon)}(b(\epsilon)), $$ where $A(\epsilon) = A + \epsilon E$ and $b(\epsilon) = b + \epsilon v$ with the columns of $A$ and $b$ orthogonal to the columns of $E$ and $v$. Let $f_1, \hdots, f_p$ and $v_1, \hdots, v_p$ denote the columns of $A$ and $E$ respectively. Then the coefficients $x_i(\epsilon)$ of the vector $x(\epsilon)$ satisfy \begin{equation} \pi_{A(\epsilon)}(b+ \epsilon v) = \sum_{i=1}^p x_i(\epsilon) (f_i + \epsilon v_i). \label{coefficientequation} \end{equation}  The orthogonality assumptions ensure that $\pi_{A(\epsilon)}(b+ \epsilon v) = \pi_{A(\epsilon)}(\overline{b} + \epsilon \overline{v})$, where $\overline{b} = \pi_{A}(b)$ and $\overline{v} = \pi_{E}(v)$, by \eqref{keyequality}. We can therefore assume without loss of generality that $b$ and $v$ lie in the column spaces of $A$ and $E$ respectively (if not we just work with $\overline{b}$ and $\overline{v}$ instead).

 By replacing some of the $f_i$ and $v_i$ by their negatives if necessary, we can assume that $b + \epsilon v$ sits inside the positive orthant, i.e.\ that $x_i(\epsilon) \geq 0$ for $1 \leq i \leq p$. Our aim is to show that the coefficients $x_i(\epsilon)$ are bounded. 
 This is enough to conclude that $\operatorname{lim}_{\epsilon \to 0}(x_i(\epsilon))$ exists for each $i$, by the following argument. As $A(\epsilon)$ has full rank for each $\epsilon \neq 0$, we have the following formula for $x(\epsilon)$ when $\epsilon \neq 0$: $$x(\epsilon) = (A(\epsilon)\T  A(\epsilon))^{-1} A(\epsilon)\T  \pi_{A(\epsilon)}(b(\epsilon)).$$ Since the entries of the vector on the right-hand side are rational functions in $\epsilon$, so are the coefficients of $x(\epsilon)$. But a bounded rational function in $\epsilon$ has a finite limit as $\epsilon \to 0$. 

To prove that the coefficients $x_i(\epsilon)$ are bounded above we start by applying the projection $\pi_A^{\perp} : = \pi_{\langle  f_1, \hdots,  f_p \rangle^{\perp}}$ to both sides of \eqref{coefficientequation}. This yields  
\begin{equation} \pi_{A}^{\perp} (\pi_{A(\epsilon)}(b+ \epsilon v)) = \sum_{i=1}^p x_i(\epsilon) \pi_{A}^{\perp}(f_i + \epsilon v_i) = \epsilon \sum_{i=1}^p x_i(\epsilon) \pi_{A}^{\perp}(v_i). \label{LambdajMLEprojected} \end{equation} Our aim is to show that the left-hand side of \eqref{LambdajMLEprojected} is bounded above by a quantity proportional to $\epsilon$ for small enough $\epsilon$. 
We use the following two lemmas. 

\begin{lemma} \thlabel{lemma1} 
    Let $P,S$ be subspaces of $\KK^n$ with $P \subseteq S$. 
    Then projections $\pi_S$ and $\pi_{P^{\perp}}$ commute. 
\end{lemma}
\begin{proof}
    If $U,V$ are subspace of $\KK^n$ satisfying $V^\perp \subset U$, then: 
   $$\pi_U \circ \pi_{V} = \pi_{(U\cap V) \oplus V^{\perp}} \circ \pi_{V} = \pi_{U\cap V} \circ \pi_{V} + \pi_{V^{\perp}} \circ \pi_{V} = \pi_{U\cap V}.$$ Therefore the composition of projections is a projection itself (onto $U \cap V$), so the projections commute. Applying this result with $U = S$ and $V = P^{\perp}$ yields the desired result.  
\end{proof}

\begin{lemma} \thlabel{lemma2} 
   Let $S,S' \in G(p,n)$ denote two subspaces of $\KK^n$ of dimension $p$. Let \[d(S,S')=\sup_{u\in S, |u|=1} d(u,S')\] where $d(u,S')$ is the distance between $u$ and $S'$. Then if $d(S,S') < \epsilon <1/4$, for any unit vector $v \in \KK^n$ we have $|\pi_S(v) -\pi_{S'}(v)|<4\epsilon$.    
\end{lemma}
\begin{proof} Assume that $d(S,S') < \epsilon < 1/4$. Let $|v|=1$ be a unit vector, and let ${S'}^\perp(v) \subset \KK^n$ be the affine space orthogonal to $S'$ and passing through $v$. First assume that $\pi_{S'}(v)$ and $\pi_{S}(v)$ are both non-zero, and let $w \in S \cap {S'}^\perp(v)$ be any point in the intersection $S\cap {S'}^\perp(v)$, see Figure~\ref{figure:drawing}. Observe that $\pi_{S'}(v) = \pi_{S'}(w)$. Then 
\begin{equation}\label{triangleineq}
    |\pi_S(v)-\pi_{S'}(v)|\le |\pi_{S'}(v)-w|+|\pi_{S}(v)-w|= 
    d(w,S')+|\pi_{S}(v)-w|.
\end{equation}

\begin{figure}[ht!]
\centering
\includegraphics[width=70mm]{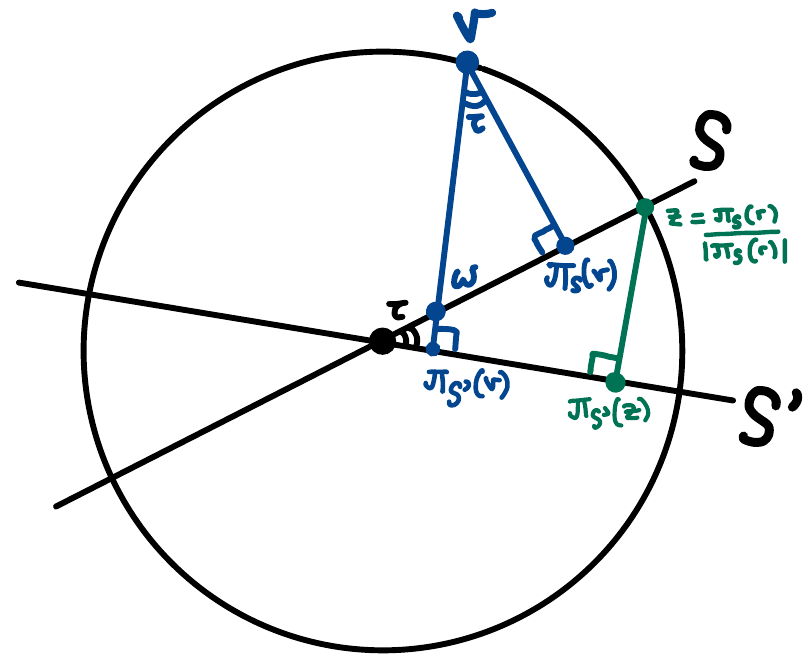}
\caption{Sketch of the set-up}
\label{figure:drawing}
\end{figure}

We wish to bound the two terms on the right-hand side of \eqref{triangleineq}. First we show that $d(w,S') < 2 \epsilon$. Let $\tau$ denote the angle formed by the subspaces $S$ and $S'$. For $\epsilon$ small enough we have $\cos (\tau) \neq 0$. Then $$ |w| =  \frac{|\pi_{S'}(w)|}{\cos(\tau)} = \frac{|\pi_{S'}(v)|}{\cos (\tau)} \leq \frac{1}{\cos(\tau)}.$$
So by choosing $\epsilon$ sufficiently small we can ensure that $|w|<2$. By definition of $d(S,S')$, $$ d(S,S') \geq d \left( \frac{w}{|w|},S' \right) = \frac{d(w,S')}{|w|}.$$ Therefore $d(w,S') = |w| d(S,S') < 2 d(S,S') < 2 \epsilon$. 
Next we show that $|\pi_S(v) - w| < 2 \epsilon$. Let $z :=\frac{\pi_S(v)}{|\pi_S(v)|}$ be the unit vector in direction $\pi_S(v)$. The two triangles 
\[(v,w,\pi_{S}(v)) \text{ and } (0,z,\pi_{S'}(z))\]
are similar as they have two equal angles -- see Figure \ref{figure:drawing}. Their scaling ratio is 
\[\frac{|v-w|}{|z|}=|v - w| \leq |v| + |w| = |w| < 2.\]
By similarity of the triangles,
\[ \frac{|\pi_S(v) - w | }{|\pi_{S'} (z) - z|} = |v-w|\] and hence $|\pi_{S}(v)-w| <  2|z-\pi_{S'}(z)|=2d(z,S)\le 2d(S,S')=2\epsilon.$ 
Hence from \eqref{triangleineq} we have $$ |\pi_S(v) - \pi_{S'}(v) | \leq d(w, S') + |\pi_S(v) - w| < 2 \epsilon + 2 \epsilon  = 4 \epsilon. $$
Finally, we consider the case where $\pi_{S'}(v)=0$ (the case $\pi_{S}(v)=0$ is similar). Assume first that $\pi_{S'}(v)=0$. Then $|\pi_S(v)-\pi_{S'}(v)|=|\pi_S(v)|$, and the triangles $(0,v,\pi_S(v))$ and $(0,z,\pi_{S'}(z))$ are congruent (isometric). Hence 
\[|\pi_S(v)|=|z-\pi_{S'}(z)|=d(z,S')<\epsilon. \qedhere \]
\end{proof}

We use \thref{lemma1,lemma2} to show that the left-hand side of \eqref{LambdajMLEprojected} is bounded above by a quantity proportional to $\epsilon$. Let $L_{\epsilon}: = \langle  f_1+\epsilon v_1, \hdots,  f_p + \epsilon v_p \rangle$ and let $L_0$ denote the limit of $L_{\epsilon}$ as $\epsilon \to 0$ in the Grassmannian $G(p,n)$ of $p$-dimensionsal subspaces of $\RR^n$. Since $\langle  f_1, \hdots,  f_p \rangle \subseteq L_0$, we can apply \thref{lemma1} with $P=\langle  f_1, \hdots,  f_p \rangle$ and $S= L_0$ to obtain \begin{equation} \pi_{A}^{\perp} (\pi_{L_0}(b+ \epsilon v)) = \pi_{L_0}(\pi_A^{\perp}(b+ \epsilon v)) = \epsilon \pi_{L_0}(\pi_{A}^{\perp}(v)), \label{commutingprojections}
\end{equation} where the second equality follows from $\pi_{A}^{\perp}(b) = 0$ (since $b$ lies in the column space of $A$ by assumption). 
Therefore by \eqref{LambdajMLEprojected} we have: \begin{align} \epsilon \sum_{i=1}^p x_i(\epsilon) \pi_A^{\perp} (v_i) & = \pi_A^{\perp} (\pi_{A(\epsilon)}(b+ \epsilon v)) \notag \\
& =  \pi_A^{\perp} (\pi_{A(\epsilon)}(b+ \epsilon v)) + \pi_A^{\perp}(\pi_{L_0}(b+ \epsilon v)) -  \pi_A^{\perp}(\pi_{L_0}(b+ \epsilon v)) \notag \\
& = \pi_A^{\perp} (\pi_{A(\epsilon)}(b+ \epsilon v)) + \epsilon \pi_{L_0}(\pi_A^{\perp}(v)) -\pi_A^{\perp}(\pi_{L_0}(b+ \epsilon v))  \text{ by \eqref{commutingprojections}} \notag \\
& = \epsilon \pi_{L_0}(\pi_A^{\perp}(v)) + \pi_A^{\perp} ((\pi_{A(\epsilon)}(b+ \epsilon v)) - \pi_{L_0}(b+ \epsilon v)). \label{finalequation}
\end{align}
We can use \thref{lemma2} to bound the norm of the difference $\pi_{A(\epsilon)}(b+ \epsilon v) - \pi_{L_0}(b+ \epsilon v) = \pi_{L_{\epsilon}}(b+ \epsilon v)) - \pi_{L_0}(b+ \epsilon v)$, by setting $S=L_{\epsilon}$ and $S' = L_0$. Note that by definition
\[d(L_{\epsilon},L_0) < d(f_i+\epsilon v_i,f_i)=\epsilon |v_i|\]
for $1\le i \le p$ and hence
\[d(L_{\epsilon},L_0)<\epsilon \max_i|v_i|\]
Then provided that $\epsilon \max_i|v_i|<\frac{1}{4}$ 
we obtain
$$ \left| \pi_{L_{\epsilon}}\left( \frac{b+ \epsilon v}{| b+ \epsilon v|}\right) - \pi_{L_0} \left( \frac{b+ \epsilon v}{| b+ \epsilon v|}\right) \right| < 4 \epsilon \max_i|v_i|$$ by \thref{lemma2}. Hence \begin{equation} |\pi_{A(\epsilon)}(b+\epsilon v) - \pi_{L_0}(b+ \epsilon v)| < 4 \epsilon | b+ \epsilon v|\max_i|v_i| < 4 \epsilon (|b|+ |v|)\max_i|v_i|. \label{inequality2}
\end{equation} 
Projecting onto the orthogonal complement of the column space of $A$ can only decrease the norm, hence we obtain from \eqref{LambdajMLEprojected}, \eqref{finalequation} and \eqref{inequality2} that \begin{equation} \pi_{A}^{\perp} (\pi_{A(\epsilon)}(b+ \epsilon v)) = \epsilon \sum_{i=1}^p x_i(\epsilon) | \pi_A^{\perp} (v_i)| < \epsilon |\pi_{L_0}(\pi_A^{\perp}(v))| + 4 \epsilon (|b|+|v|)\max_i|v_i|. \label{normestimation} \end{equation} 
Since $x_i(\epsilon) \geq 0$, it follows that the $x_i(\epsilon)$ are bounded above, with bound
\[x_i(\epsilon)\le \frac{|\pi_{L_0}(\pi_{A}^{\perp}(v))|+4(|b| + |v|)\max_i|v_i|}{\max_{1 \leq i \leq p} |\pi_A^\perp(v_i)|}.\]

\subsection{Algebraic proof}
We now prove the first part of Theorem~\ref{thm:linear_system} again, via an algebraic argument. The algebraic approach has the advantage that it also gives an explicit description of the limit. We assume in this section that $\KK= \RR$. The proof over $\KK = \CC$ is identical as long as we replace the transpose by the conjugate transpose. 

Let $A(\epsilon) = A+ \epsilon E $ and $b(\epsilon) = b + \epsilon v$ be as in \thref{thm:linear_system}.  Then the unique solution $x(\epsilon)$ to $A(\epsilon) x(\epsilon) = \pi_{A(\epsilon)} (b(\epsilon))$ is given by $$ x(\epsilon) = A(\epsilon)^+ \pi_{A(\epsilon)} (b+\epsilon v).$$ Let $f_i$ denote the columns of $A$ and $v_i$ the columns of $E$. Define $\overline{b} = \pi_{A }(b)$ and $\overline{v} = \pi_{E}(v)$. Then $ \pi_{A(\epsilon)} ( b(\epsilon)) = \pi_{A(\epsilon)} (\overline{b} + \epsilon \overline{v}),$ by \eqref{keyequality}.

Since $A(\epsilon)$ has full column rank by assumption for $\epsilon \neq 0$, its pseudo-inverse is $$ A(\epsilon)^{+} = (A(\epsilon)\T(A(\epsilon)))^{-1} A(\epsilon)\T.$$ Therefore \begin{align*}  x(\epsilon) & = (A(\epsilon)\T A(\epsilon))^{-1} A(\epsilon)\T  A(\epsilon) (A(\epsilon)\T  A(\epsilon))^{-1} A(\epsilon)\T  (\overline{b} + \epsilon \overline{v}) \\
& =  (A(\epsilon)\T A(\epsilon))^{-1} A(\epsilon)\T  (\overline{b} + \epsilon \overline{v}).
\end{align*} 
Define $C(\epsilon) := A(\epsilon)\T A(\epsilon).$  Note that $C(\epsilon) = A\T A + \epsilon^2 E\T  E$, since the columns of $A$ are orthogonal to those of $E$. Then \begin{equation} x(\epsilon) = C(\epsilon)^{-1} A(\epsilon)\T  (\overline{b} + \epsilon \overline{v}) =  \frac{1}{\det C(\epsilon)} \adj  C(\epsilon) \begin{pmatrix} 
 f_1 \cdot \overline{b} + \epsilon^2 v_1 \cdot \overline{v} \\
\vdots \\
 f_p \cdot \overline{b} + \epsilon^2 v_p \cdot \overline{v}
\end{pmatrix}.\label{equationforxe} \end{equation}

We seek an expression for $x(\epsilon)$ without powers of $\epsilon$ in the denominator. 
We begin with the case  where $\overline{b} + \overline{v}$ lies in $\langle f_i + v_i : 1 \leq i \leq p \rangle$. This is the same assumption as in \thref{MLEofstab}.  

\begin{lemma} \thlabel{simplecase}
Suppose $\overline{b} + \overline{v} \in \langle f_i + v_i :1 \leq i \leq p \rangle, $ so that $\overline{b} + \overline{v} = \sum_{i = 1}^p \mu_i (f_i + v_i)$ for some $\mu_i \in \KK$.
Let $e_i$ be the $i$-th standard basis vector in $\KK^p$. 
Then for all $\epsilon$ we have
$x(\epsilon) =    \sum_{i =1}^p \mu_i e_i$.
\end{lemma}

\begin{proof}
We can assume that $\overline{b} + \overline{v} = f_i + v_i$ for some $1 \leq i \leq p$.
The case where $\overline{b} + \overline{v}$ is a general linear combination follows similarly.
We want to prove that $x(\epsilon) = e_i$. We give both an algebraic and a geometric proof of this result. We start with the algebraic proof.

From \eqref{equationforxe} we see that the $i$-th entry of $x(\epsilon)$ is $1/\det C(\epsilon)$ times  
the dot product of the $i$-th row of $\adj C(\epsilon)$ with the $i$-th column of $C(\epsilon)$.
The latter is $\det C(\epsilon)$,
in its cofactor expansion along the $i$-th column of $C(\epsilon)$. Hence the $i$-th entry of $x(\epsilon)$ is $1$. Now take $i \neq l \in \{1, \hdots, p \}$. The $l$-th entry of $x(\epsilon)$ is $1/\det(\epsilon)$ times the dot product of the $l$-th row of $\adj C(\epsilon)$ with the $i$-th column of $C(\epsilon)$. Since $l \neq i$ this is a cofactor expansion using a different column, and therefore the expression vanishes. Hence $x(\epsilon) = e_i$. 

The geometric proof is as follows. The entries of $x(\epsilon)$ are the coefficients in front of $\{f_i + \epsilon v_i :1 \leq i \leq p\}$ in the projection of $\overline{b} + \epsilon \overline{v}$ to $\langle f_i + \epsilon v_i : 1 \leq i \leq p \rangle$. Since we are assuming $\overline{b} +  \overline{v} = f_i + v_i$, we have $\overline{b}= f_i$ and $\overline{v} = v_i$ by the orthogonality assumptions. Therefore $$\pi_{\langle f_i + \epsilon v_i : i \in \{1, \hdots, p\} \rangle}(\overline{b}+ \epsilon \overline{v}) = \pi_{\langle f_i + \epsilon v_i : i \in \{1, \hdots, p\} \rangle} (f_i + \epsilon v_i) = f_i + \epsilon v_i.$$ So $x(\epsilon) = e_i$. 
\end{proof} 

We now turn to the general case.
Since $\epsilon$ only appears in $x(\epsilon)$ with even powers, to simplify calculations we let \begin{equation} C'(\epsilon) = A\T  A + \epsilon E\T E 
\label{Cprime}
\end{equation} and consider the following vector: \begin{equation}  
x'(\epsilon) = \frac{1}{\det C'(\epsilon)} \adj C'(\epsilon)  \begin{pmatrix} 
 f_1 \cdot \overline{b} + \epsilon v_1 \cdot \overline{v} \\
\vdots \\
 f_p \cdot \overline{b} + \epsilon v_p \cdot \overline{v}
\end{pmatrix}. \label{optimalexpression2}
\end{equation}
Note that $\lim_{\epsilon \to 0} x'(\epsilon)$ exists if and only if $\lim_{\epsilon \to 0} x(\epsilon)$ exists, since $x(\epsilon) = x'(\epsilon^2)$.
We expand the polynomial $\det C'(\epsilon)$ as
$$ \det C'(\epsilon) = c_0 + c_1 \epsilon + c_2 \epsilon^2 + \cdots + c_{p} \epsilon^{p}$$ for some coefficients $c_i \in \KK$. Similarly, we write 
$$\adj C'(\epsilon)\begin{pmatrix} 
 f_1 \cdot \overline{b} + \epsilon v_1 \cdot \overline{v} \\
\vdots \\
 f_p \cdot \overline{b} + \epsilon v_p \cdot \overline{v}
\end{pmatrix} = D_0 + D_1 \epsilon  + \cdots + D_{p} \epsilon^{p}$$
for some column vectors $D_i$, 
since the entries are 
polynomials in $\epsilon$.
Then
\begin{equation} x'(\epsilon) = \frac{D_0 + D_1 \epsilon  + \cdots + D_{p} \epsilon^{p}}{ c_0 + c_1 \epsilon + \cdots + c_{p} \epsilon^{p}}.  \label{secondpolynomial}
\end{equation}
We see that $\operatorname{lim}_{\epsilon \to 0} x'(\epsilon)$ exists if and only if whenever $c_{k} = 0$ for all $k \leq l$ (for some $0 \leq l \leq p)$, we have 
 $D_{k}= 0$ for all $k \leq l$.
 
We now describe the coefficients $c_i$. First, 
$$ c_0 = \det C(0) = \det A\T A,$$ since to obtain the constant term in $\epsilon$, only the matrix $A\T A$ need be considered. We use Jacobi's formula for the derivative of a determinant to calculate
\begin{align*} 
c_1 & = \left.\frac{d}{d \epsilon}\right|_{\epsilon =0} \det C'(\epsilon)  =  \operatorname{tr}  \left( \adj C'(\epsilon) \frac{d}{d \epsilon} C(\epsilon) \right)(0) \\
 & = \operatorname{tr} ( (\adj C'(0) )E\T E ) = \operatorname{tr} ( (\adj A\T A) E\T E) .
\end{align*} 
We apply Jacobi's formula again to compute $c_2$:
$$ 
c_2  = \left.\frac{d^2}{d \epsilon^2}\right|_{\epsilon =0} \det C'(\epsilon) = \left.\frac{d}{d \epsilon}\right|_{\epsilon =0}  \left( \frac{d}{d \epsilon} \det C'(\epsilon) \right)  = \left.\frac{d}{d \epsilon}\right|_{\epsilon =0} \operatorname{tr} ( (\adj C'(\epsilon)) E\T  E ). 
$$ 
Proceeding in this way we obtain  \begin{equation} c_i = \operatorname{tr} \left(  \left( \left. \frac{d^{i-1}}{d \epsilon^{i-1}} \right|_{\epsilon = 0}  \adj C'(\epsilon) \right)  E\T  E \right), \label{ci} 
\end{equation} 
for all $i \in \{1, \hdots, p\}$.
We now turn to the coefficients $D_i$. Expanding \eqref{secondpolynomial} gives
\begin{equation}  
D_0 + D_1 \epsilon + \cdots + D_p \epsilon^p = \adj C'(\epsilon) \begin{pmatrix} 
 f_1 \cdot \overline{b} \\
\vdots \\
 f_p \cdot \overline{b}
\end{pmatrix} + \epsilon \adj C'(\epsilon) \begin{pmatrix} v_1 \cdot \overline{v} \\
\vdots \\
v_p \cdot \overline{v}
\end{pmatrix}. \label{calculatingDi}
\end{equation} 
It follows that $$D_0 = \adj C'(0) \begin{pmatrix}  f_1 \cdot \overline{b} \\
\vdots \\
 f_p \cdot \overline{b}
\end{pmatrix} = \adj A\T  A \begin{pmatrix}  f_1 \cdot \overline{b} \\
\vdots \\
 f_p \cdot \overline{b}
\end{pmatrix} .$$  The coefficient $D_1$ is the sum of the degree $1$ part of $\adj C'(\epsilon)$ multiplied by the vector with entries $f_i \cdot \overline{b}$, and the degree $0$ part of $\adj C'(\epsilon)$ multiplied by the vector with entries $v_i \cdot \overline{v}$. Therefore we have 
\begin{equation}  D_1 =  \left. \frac{d}{d \epsilon} \right|_{\epsilon =0} \adj C'(\epsilon) \begin{pmatrix}  f_1 \cdot \overline{b} \\
\vdots \\
 f_p \cdot \overline{b}
\end{pmatrix} + \adj C'(0) \begin{pmatrix} v_1 \cdot \overline{v} \\
\vdots \\
v_p \cdot \overline{v}
\end{pmatrix}. \label{D1}  
\end{equation} 
Proceeding in this way we obtain for all $i \in \{1, \hdots, p\}$ that
\begin{equation} D_i = \left. \frac{d^i}{d \epsilon^i} \right|_{\epsilon = 0} \adj C'(\epsilon)\begin{pmatrix}  f_1 \cdot \overline{b} \\
\vdots \\
 f_p \cdot \overline{b}
\end{pmatrix} + \left. \frac{d^{i-1}}{d \epsilon^{i-1}} \right|_{\epsilon = 0} \adj C'(\epsilon)\begin{pmatrix} v_1 \cdot \overline{v} \\
\vdots \\
v_p \cdot \overline{v}
\end{pmatrix}. \label{Di} \end{equation}

We want to show that if $c_i = 0$ for all $i \leq l$ (for some $0 \leq l \leq p$), then $D_i =0$ for all $ i \leq l$. The following lemma achieves this. Indeed, conditions \ref{conda} and \ref{condb} together ensure that $D_i = 0$ for all $i \leq l$, based on the expression for $D_i$ given in \eqref{Di} above.

\begin{lemma} \thlabel{stronginduction}
Fix $0 \leq l \leq p$ and suppose that $c_i = 0$ for all $i \leq l$. Then \begin{enumerate}[(a)]
\item \label{conda} $\left. \frac{d^i}{d \epsilon^{i}} \right|_{\epsilon = 0} \adj C'(\epsilon) = 0$ for all $ i < l$;  
\item \label{condb}   $\left. \frac{d^l}{d \epsilon^{l}} \right|_{\epsilon =0} \adj C'(\epsilon) A\T  A = 0.$
\end{enumerate}
\end{lemma} 

\begin{proof}
We use strong induction. We start with the base case $l = 0$, so we assume that $c_0 = 0$. For \ref{conda} there is nothing to check since $l = 0$. To show \ref{condb}, it is enough to show that  $$ \adj C'(0)  \begin{pmatrix} 
 f_1 \cdot  f_k \\
\vdots \\
 f_p \cdot  f_k
\end{pmatrix} = 0,$$
for each $k \in \{1, \hdots, p\}$.
Now $\adj C'(0) = A\T A$, and we have that $$ \adj A\T  A \begin{pmatrix} 
 f_1 \cdot  f_k \\
\vdots \\
 f_p \cdot  f_k
\end{pmatrix} =  (0, \hdots, 0 , \det A\T  A ,0,\hdots, 0)\T ,$$ where $\det A\T A$ appears in the $k$-the entry, by the same cofactor expansion argument as in the proof of \thref{simplecase}. Since $c_0 = \det A\T A =0$, it follows that the above expression vanishes, which shows \ref{condb} when $l = 0$. 
This establishes the base case.

Fix some $1 \leq l \leq p$ and suppose that $c_i = 0$ for all $i \leq l$. Assume:
\begin{enumerate}
\item[$(a_{l-1})$]  $\left. \frac{d^k}{d \epsilon^{k}} \right|_{\epsilon = 0} \adj C'(\epsilon) = 0$ for all $ k < l-1$; 
\item[$(b_{l-1})$]  $\left. \frac{d^{l-1}}{d \epsilon^{l-1}} \right|_{\epsilon =0} \adj C'(\epsilon) A\T  A = 0.$
\end{enumerate}
We wish to show that: 
\begin{enumerate}
\item[$(a_{l})$]  $\left. \frac{d^{l-1}}{d \epsilon^{l-1}} \right|_{\epsilon = 0} \adj C'(\epsilon) = 0$;
\item[$(b_{l})$]  $\left. \frac{d^{l}}{d \epsilon^{l}} \right|_{\epsilon =0} \adj C'(\epsilon) A\T  A = 0.$
\end{enumerate}
We start by proving $(a_l)$. By $(b_{l-1})$ we know that $\left. \frac{d^{l-1}}{d \epsilon^{l-1}} \right|_{\epsilon =0} \adj C'(\epsilon) A\T  A = 0,$ therefore it is sufficient to show that \begin{equation} \label{whatwewant} \left. \frac{d^{l-1}}{d \epsilon^{l-1}} \right|_{\epsilon =0} \adj C'(\epsilon) E\T  E = 0, \end{equation} since $A\T A + E\T E = (A+E)\T (A+E)$ is invertible. We now prove \eqref{whatwewant} using the assumptions $c_l = 0$ and $(a_{l-1})$. 
Using the expression for $c_l$ given in \eqref{ci} we have
\begin{equation} \operatorname{tr} \left(\left.  \frac{d^{l-1}}{d \epsilon^{l-1}} \right|_{\epsilon = 0} \adj C'(\epsilon)  E\T  E \right) = 0. \label{tracezero} 
\end{equation} 
If $ \left.  \frac{d^{l-1}}{d \epsilon^{l-1}} \right|_{\epsilon = 0} \adj C'(\epsilon)$ is positive semi-definite, then the product inside the trace in~\eqref{tracezero} is zero. This is because $E\T  E$ is positive semi-definite and the trace of a product of positive semi-definite matrices is zero if and only if the product is zero. We can establish that $ \left.  \frac{d^{l-1}}{d \epsilon^{l-1}} \right|_{\epsilon = 0} \adj C'(\epsilon)$ is positive semi-definite using $(b_{l-1})$.
By $(b_{l-1})$, we know that \begin{align*}  \adj C'(\epsilon) &  = \epsilon^{l-1} \left( \left. \frac{d^{l-1}}{d \epsilon^{l-1}} \right|_{\epsilon = 0} \adj C'(\epsilon) \right) + \epsilon^l \left( \left. \frac{d^{l}}{d \epsilon^{l}} \right|_{\epsilon = 0} \adj C'(\epsilon) \right) + \cdots \\
& = \epsilon^{l-1} \left(  \left( \left. \frac{d^{l-1}}{d \epsilon^{l-1}} \right|_{\epsilon = 0} \adj C'(\epsilon) \right) + \epsilon \left( \left. \frac{d^{l}}{d \epsilon^{l}} \right|_{\epsilon = 0} \adj C'(\epsilon) \right) + \cdots \right) . \label{cancelling} 
\end{align*} 
The matrix $C'(\epsilon)$ is positive semi-definite. Hence $\operatorname{lim}_{\epsilon \to 0} \epsilon^{-l+1} \adj C'(\epsilon)$ is positive semi-definite. But this limit is
$$ \left. \frac{d^{l-1}}{d \epsilon^{l-1}} \right|_{\epsilon = 0} \adj C'(\epsilon).$$ 
Hence the product inside the trace in~\eqref{tracezero} is zero.
This proves $(a_l)$. 
To show $(b_l)$, it is enough to show that for any $k \in \{1, \hdots, p\}$ we have $$\left. \frac{d^l}{d \epsilon^l} \right|_{\epsilon = 0} \adj C'(\epsilon)\begin{pmatrix}  f_1 \cdot  f_k \\
\vdots \\
 f_p \cdot  f_k
\end{pmatrix}  = 0. $$ 
Expanding the expression on the left hand side gives
\begin{align*} 
  & \left. \frac{d^l}{d \epsilon^l} \right|_{\epsilon = 0} \left(  \adj C'(\epsilon)  \begin{pmatrix} 
 f_1 \cdot  f_k + \epsilon v_1 \cdot v_k \\
\vdots \\
 f_p \cdot  f_k + \epsilon v_p \cdot v_k \end{pmatrix} \right) -  \epsilon \left( \left. \frac{d^{l-1}}{d \epsilon^{l-1}} \right|_{\epsilon = 0} \adj C'(\epsilon) \right) \begin{pmatrix}
v_1 \cdot v_k \\
\vdots \\
v_p \cdot v_k 
\end{pmatrix} \\
  = & \left. \left( \frac{d^l}{d \epsilon^l}\right|_{\epsilon = 0} \det C'(\epsilon)\right) e_k  - \epsilon \left( \left. \frac{d^{l-1}}{d \epsilon^{l-1}} \right|_{\epsilon =0} \adj C'(\epsilon) \right) \begin{pmatrix} 
 v_1 \cdot v_k \\
 \vdots \\
 v_p \cdot v_k
 \end{pmatrix} ,
\end{align*} 
This is zero, because 
$\left. \frac{d^l}{d \epsilon^l} \right|_{\epsilon = 0} \det C'(\epsilon) = 0$
and the product inside the trace in~\eqref{tracezero} is zero.
 This proves $(b_l)$.     \end{proof}

\begin{corollary}
\thlabel{cor:limit_formula}
Let $A(\epsilon) = A + \epsilon E$ and $b(\epsilon) = b+ \epsilon v$ be as in \thref{thm:linear_system}. Let $\{  f_1, \ldots,  f_p \}$ be the columns of $A$ and $\{ v_1, \ldots, v_p \}$ the columns of $E$. 
Define $\overline{b} = \pi_A (b)$ and $\overline{v} = \pi_E(v)$. 
Let $C'(\epsilon)$ be defined as at \eqref{Cprime}.
Let $x(\epsilon)$ be the unique solution to $A(\epsilon) x(\epsilon) = \pi_{A(\epsilon)}(b(\epsilon))$. Then the limit $\lim_{\epsilon \to 0} x(\epsilon)$ exists and equals 
$$  \frac{1}{\operatorname{tr} \left( \left.  \frac{d^{l-1}}{\epsilon^{l-1}} \right|_{\epsilon = 0}  \adj C'(\epsilon) E\T E \right)} \left( \left. \frac{d^l}{d \epsilon^l} \right|_{\epsilon = 0} \adj C'(\epsilon)\begin{pmatrix}  f_1 \cdot \overline{b} \\
\vdots \\
 f_p \cdot \overline{b}
\end{pmatrix} + \left. \frac{d^{l-1}}{d \epsilon^{l-1}} \right|_{\epsilon = 0} \adj C'(\epsilon)\begin{pmatrix} v_1 \cdot \overline{v_j} \\
\vdots \\
v_p \cdot \overline{v_j}
\end{pmatrix}  \right) ,$$ 
where  $l$ denotes the smallest integer in $\{1, \hdots, p\}$ with  $\operatorname{tr} \left( \left.  \frac{d^{l-1}}{\epsilon^{l-1}} \right|_{\epsilon = 0}  \adj C'(\epsilon) E\T E \right) \neq 0$.  
\end{corollary} 

\begin{proof}
Since $c_i =0$ for all $i <l$, by \thref{stronginduction} we have $D_i = 0$ for all $i <l$ (based on the expression for $D_i$ given in \eqref{Di}), so that $$ \mu_j(\epsilon) = \frac{1}{c_l \epsilon^{2l} + c_{l+1} \epsilon^{2l+2} + \cdots} \left( D_l \epsilon^{2l} + D_{l+1} \epsilon^{2l+2} + \cdots \right) = \frac{D_l}{c_l} + \frac{D_{l+1}}{c_{l+1}} \epsilon^2 + \cdots.$$ Therefore the limit of $\mu_j(\epsilon)$ as $\epsilon$ tends to zero exists, and equals $ D_l/ c_l$. \end{proof}

\section{MLEs given sample stabilisations in the limit} \label{sec:uniqueMLEinlimit} 

We gave necessary and sufficient conditions for the MLE given an $f$-stabilisation to be an MLE given $f$ in Section \ref{sec:MLEofsamplestab}. In this section we consider the limit of the $\Lambda$-MLE or MLE given $\widetilde{f}(\epsilon) : = f+ \epsilon f'$ as $\epsilon \to 0$. We show that we always obtain an MLE given $f$ (if one exists, otherwise a $\Lambda$-MLE), in Section \ref{subsec:limitMLE}. We study which MLEs given $f$ can be obtained as MLEs given $f$-stabilisations under such a limit, in Section \ref{subsec:whichMLEslimit}.

\subsection{The limit MLE given $\widetilde{f}$ exists and is an MLE given $f$} \label{subsec:limitMLE}

We prove that if $\widetilde{f} = f+ f'$ is an $f$-stabilisation, then the MLE given $\widetilde{f}(\epsilon) : =f+ \epsilon f'$ has a well-defined limit as $\epsilon$ tends to zero, and moreover that this limit is an MLE given $f$ if one exists. If the MLE given $f$ does not exist, then the previous statement remains true by considering the $\Lambda$-MLE. We also describe the $\Lambda$-MLE and MLE given $f$ that is picked out by this process. We start by proving the result about $\Lambda$-MLEs, before turning to MLEs in \thref{mainresult}. 

\begin{proposition}[Limit $\Lambda$-MLE given an $f$-stabilisation] \thlabel{mainresultLambda}
Fix a DAG $\mathcal{G}$, a sample $f$ and an $f$-stabilisation $\widetilde{f}= f + f'$. Let $f_i$ denote the columns of $f$ and $v_i$ the columns of $f'$. For each child vertex $i$ let $C'_i(\epsilon) : = A_i\T A_i + \epsilon E_i\T E_i$, where $A_i$ (respectively $E_i$) is the $n \times m$ matrix with columns the subset of the $f_j$ (respectively $v_j$) such that $j \to i$. Let $\overline{f_i}= \pi_{\langle f_j :j \to i \rangle}(f_i)$ and let $\overline{v_i} = \pi_{\langle v_j :j \to i \rangle}( v_i)$. Let $\widetilde{f}(\epsilon) = f + \epsilon f'$ for $\epsilon \neq 0$. Then we have the following results about MLEs in the DAG model on $\Gcal$: 
\begin{enumerate}[(a)]
    \item \label{uniqueepsilon} a unique $\Lambda$-MLE exists given $\widetilde{f}(\epsilon)$ for any $\epsilon \neq 0$; 
    
    \item \label{existenceoflimit} fix a vertex $i$ and suppose for simplicity that $ f_1, \hdots,  f_p$ and $v_1, \hdots, v_p$ are the columns of $f$ and $f'$ respectively indexed by edges $j \to i$. Then the $\Lambda_i$-MLE given $\widetilde{f}(\epsilon)$ has a well-defined limit as $\epsilon$ tends to zero, given by: 
        $$ \frac{1}{\operatorname{tr} \left( \left.  \frac{d^{l-1}}{\epsilon^{l-1}} \right|_{\epsilon = 0}  \adj C_i'(\epsilon) E\T  E \right)} \left( \left. \frac{d^l}{d \epsilon^l} \right|_{\epsilon = 0} \adj C_i'(\epsilon)\begin{pmatrix}  f_1 \cdot \overline{f_i} \\
        \vdots \\
         f_p \cdot \overline{f_i}
        \end{pmatrix} + \left. \frac{d^{l-1}}{d \epsilon^{l-1}} \right|_{\epsilon = 0} \adj C_i'(\epsilon)\begin{pmatrix} v_1 \cdot \overline{v_i} \\
        \vdots \\
        v_p \cdot \overline{v_i}
        \end{pmatrix}  \right), $$
    where  $l \in \{1, \hdots, p\}$ denotes the smallest integer such that
        $$ \operatorname{tr} \left( \left.  \frac{d^{l-1}}{\epsilon^{l-1}} \right|_{\epsilon = 0}  \adj C_i'(\epsilon) E\T  E \right) \neq 0.$$  
    
    \item \label{limit} the limit of the $\Lambda$-MLE given $\widetilde{f}(\epsilon)$ as $\epsilon$ tends to zero is a $\Lambda$-MLE given $f$;
    
    \item \label{independent} if $\overline{f_i} + \overline{v_i} \in \langle f_j + v_j :j \to i \rangle$ for all vertices $i$, then the $\Lambda$-MLE given $\widetilde{f}(\epsilon)$ is independent of $\epsilon$ and is a $\Lambda$-MLE given $f$. 
\end{enumerate}
   \end{proposition}

\begin{proof}
If $f'$ is an $f$-perturbation, then $\epsilon f'$ is also an $f$-perturbation for any $\epsilon  \neq 0$. Therefore $\widetilde{f}(\epsilon)$ is an $f$-stabilisation for any $\epsilon \neq 0$ and so by \thref{samplestabisstable} there is a unique MLE given $\widetilde{f}$. In particular there is a unique $\Lambda$-MLE given $\widetilde{f}$. This proves \ref{uniqueepsilon}.

We now turn to \ref{existenceoflimit} and \ref{limit}. 
We can find the $\Lambda$-MLE by finding each $\Lambda_i$-MLE independently.
The $\Lambda_i$-MLE given a sample $Y$ are the coefficients $\lambda_{ij}$ in front of each $Y^{(j)}$ in the orthogonal projection of $Y^{(i)}$ onto the span of $\{ Y^{(j)}  :j \to i \}$. 
Hence they are the entries of $x$ in a linear system of the form $Ax = \pi_A(b)$, where $A$ has the vectors $Y^{(j)}$ for $j \to i$ as its columns and $b= Y^{(i)}$. Therefore the $\Lambda_i$-MLE is not unique if and only if the linear system is underdetermined. 

By definition, the $\Lambda_i$-MLE given $\widetilde{f}(\epsilon)$ is the unique solution $x(\epsilon)$ to the linear system $$(A_i+ \epsilon E_i)x(\epsilon) = \pi_{\langle f_j + \epsilon v_j  : j \to i \rangle}(f_i+v_i).$$ Recall that the matrix $A_i+\epsilon E_i$ is the matrix obtained from $f+ \epsilon f'$ by picking out those columns indexed by vertices $j$ such that $j \to i$.  Since the columns of $f$ are orthogonal to the columns of $f'$ by the definition of a sample stabilisation, it follows that the columns of $A_i$ and $f_i$ are orthogonal to the columns of $E_i$ and $v_i$. We also know that $f+ \epsilon f'$ has full column rank, since $f+ \epsilon f'$ is an $f$-stabilisation. Therefore $A_i + \epsilon E_i$ has full column rank for each $\epsilon \neq 0$. We have thus shown that $A_i(\epsilon) = A_i + \epsilon E_i$ and $f_i(\epsilon) = f_i+ \epsilon v_i$ satisfy the assumptions of \thref{thm:linear_system}. It follows that $x(\epsilon)$ has a well-defined limit as $\epsilon$ tends to zero, and moreover that the limit $x(0)$ is a solution to $ A_i x = \overline{f_i}$. This ensures that the limit $x(0)$ is a $\Lambda_i$-MLE given $f$. The formula is obtained from Corollary~\ref{cor:limit_formula}.

It remains to show \ref{independent}. If $\overline{f_i} + \overline{v_i} \in \langle f_j + v_j :j \to i \rangle$, so that $\overline{f_i} + \overline{v_i} = \sum_{j \to i} \mu_j (f_j + v_j)$ for some $\mu_j \in \KK$, then the $\Lambda_i$-MLE given $\widetilde{f}(\epsilon)$ is 
$ \mu_i(\epsilon) = \sum_{j \to i} \mu_j e_j$, which is independent of~$\epsilon$, by \thref{simplecase}. Hence the limit $\Lambda$-MLE, which is a $\Lambda$-MLE given $f$, is also a $\Lambda$-MLE given $\widetilde{f}(\epsilon)$ for any $\epsilon \neq 0$. 
\end{proof}

\begin{remark}[Connection to \thref{MLEofstab}]
\thref{mainresultLambda}\ref{independent} is the reverse implication of \thref{MLEofstab}. We included it above because we prove it using a different method.  
\end{remark}

We build on \thref{mainresultLambda} to obtain an analogous result about MLEs. 

\begin{theorem}[Limit MLE given a sample stabilisation] \thlabel{mainresult} 
Fix a DAG $\mathcal{G}$, a sample $f$ and an $f$-stabilisation $\widetilde{f} : = f+ f'$. Let $\widetilde{f}(\epsilon) : = f+ \epsilon f'$ for $\epsilon \neq 0 $. Then we have the following results about MLEs in the DAG model on $\Gcal$:
\begin{enumerate}[(a)]
\item $\widetilde{f}(\epsilon)$ has a unique MLE for any $\epsilon \neq 0$; \label{parta}
\item the MLE given $\widetilde{f}(\epsilon)$ has a well-defined limit as $\epsilon \to 0$; \label{partb}
\item if $f$ has at least one MLE then the limit is an MLE given $f$, more precisely the unique MLE with $\Lambda$-MLE component given in \thref{mainresultLambda}\ref{existenceoflimit}; \label{partc}
\item the MLE given $\widetilde{f}(\epsilon)$ is independent of $\epsilon$ and an MLE given $f$ if and only if \begin{equation} v_i \in \langle v_j :j \to i \rangle \text{ and }  \overline{f_i} +  v_i \in \langle f_j + v_j :j \to i \rangle,  
\label{conditions}
\end{equation} for all child vertices $i$, where $\overline{f_i} := \pi_{\langle f_j : j \to i \rangle}(f_i)$.  \label{partd}
\end{enumerate}
    
\end{theorem}

\thref{mainresult}\ref{parta} and \ref{partc} is \thref{firstmainresult}\ref{iff}, while \thref{mainresult}\ref{partc} is \thref{firstmainresult}\ref{limitis}. 

\begin{proof}
For \ref{parta}, see the proof of \thref{mainresultLambda}\ref{uniqueepsilon}. 
By \thref{mainresultLambda}\ref{existenceoflimit}, the $\Lambda$-MLE given $\widetilde{f}(\epsilon)$ has a well-defined limit as $\epsilon \to 0$. It remains to show that the $\Omega$-MLE also has a well-defined limit. 
The $\Omega$-MLE $\omega(\epsilon)$ given $\widetilde{f}(\epsilon)$ has components
$$ \omega_{i}(\epsilon) = || \pi_{\langle f_j + \epsilon v_j :  j \to i \rangle} (f_i+  \epsilon v_i) - f_i - \epsilon v_i || .$$
We have
$$\pi_{\langle f_j + \epsilon v_j : j \to i \rangle} (f_i + \epsilon v_i) \to \pi_{\langle f_j  : j \to i \rangle}(f_i)$$ as $\epsilon \to 0$,
by the proof of Proposition~\ref{prop:if_limit}.
Since the limit commutes with taking the norm, it follows that $\omega_{i}(\epsilon)$ tends to $||  \pi_{\langle f_j :  j \to i \rangle}(f_i) - f_i ||$ as $\epsilon$ tends to zero. Therefore the $\Omega$-MLE given $\widetilde{f}(\epsilon)$ has a limit as $\epsilon$ tends to zero, which proves \ref{partb}. Moreover, if the $\Omega$-MLE given $f$ exists, then the limits $||  \pi_{\langle f_j : j \to i \rangle}(f_i) - f_i ||$ for all $j \to i$ make up the $\Omega$-MLE given $f$. Together with \thref{mainresultLambda}\ref{existenceoflimit}, this establishes \ref{partc}. 

To prove \ref{partd}, suppose first that the MLE given $\widetilde{f}(\epsilon)$ is independent of $\epsilon$ and an MLE given $f$. The second assumption ensures by \thref{OmegaMLEofstab} that the equations in \eqref{conditions} are satisfied for all child vertices $i$.
Conversely, if these equations are satisfied then they are also satisfied if the $v_i$ and $v_j$ are replaced by $\epsilon v_i$ and $\epsilon v_j$ for $\epsilon \neq 0$. Therefore the MLE given $\widetilde{f}(\epsilon)$ is an MLE given $f$, by \thref{OmegaMLEofstab}. In particular the $\Omega$-MLE given $\widetilde{f}(\epsilon)$ is the unique $\Omega$-MLE given $f$, which is independent of $\epsilon$. 
Moreover, if these equations are satisfied, then by \thref{mainresultLambda}\ref{independent} the $\Lambda$-MLE given $\widetilde{f}(\epsilon)$ is independent of $\epsilon$ and also a $\Lambda$-MLE given $f$. This shows that the MLE given $\widetilde{f}(\epsilon)$ is independent of $\epsilon$ and an MLE given $f$.
\end{proof} 

\begin{remark}[Strenghtening \thref{mainresult}\ref{partc}]
Our proof of \thref{mainresult}\ref{partc} proves a stronger statement, which doesn't require that an MLE given $f$ exists: if an MLE exists on a subset of vertices, then the limit of the partial MLE given $\widetilde{f}(\epsilon)$ on these vertices is a partial MLE given~$f$. 
\end{remark} 

We conclude this section by giving a name to the MLEs and $\Lambda$-MLEs obtained in the limit.

\begin{definition}[Limit MLE given a sample stabilisation]
Given a sample $f$ and an $f$-stabilisation $\widetilde{f} = f+ f'$, the \emph{limit $\Lambda$-MLE given $\widetilde{f}$} is the limit as $\epsilon$ tends to zero of the $\Lambda$-MLE given $\widetilde{f}(\epsilon) = f+\epsilon f'$. If the MLE given $f$ exists, then the \emph{limit MLE given $\widetilde{f}$} is defined analogously.  
\end{definition} 

\subsection{When is an MLE given $f$ the limit MLE given an $f$-stabilisation?} \label{subsec:whichMLEslimit}

We know that for a sample $f$, the limit MLE given any $f$-stabilisation is an MLE given $f$ if at least one exists given $f$, by \thref{mainresult}. In this section we address the following question: which MLEs given $f$ are limit MLEs given $f$-stabilisations? This question should be viewed as an extension of the question posed in Section \ref{subsec:whichMLEs} regarding which MLEs given $f$ coincide with the MLE given an $f$-stabilisation. We approach the question geometrically, giving an analogue of \thref{MLEfromMLEofstab}. We start first by answering the question for $\Lambda$-MLEs in \thref{answerq2} below. The solution to the problem for MLEs will follow immediately, see \thref{answerq2gen}. 

The statement of \thref{answerq2} requires defining for a $\Lambda$-MLE $\lambda$ given $f$ an associated locally closed subvariety $X_{f,\lambda}^{\operatorname{lim}}$ of the parameter space $X_f \subseteq X = \KK^{n \times m}$ of $f$-stabilisations defined in Section \ref{subsec:paramspace}. This subvariety will parametrise $f$-stabilisations $\widetilde{f}$ such that the limit $\Lambda$-MLE given $\widetilde{f}$ is $\lambda$.  To define $X_{f,\lambda}^{\operatorname{lim}}$, fix a sample $f$ and $\lambda$ a $\Lambda$-MLE given $f$. Let $\lambda_i$ denote the $\Lambda_i$-MLE for  each child vertex $i$. We represent $\lambda_i$ as a column vector of length $|\operatorname{pa}(i)|.$

By \thref{mainresultLambda}\ref{existenceoflimit} we know that for any $f$-perturbation $\widetilde{f}$ and any vertex $i$, the limit of the $\Lambda_i$-MLE given $\widetilde{f}(\epsilon) = f + \epsilon f'$ as $\epsilon$ tends to zero equals $D_l/c_l.$ 
We are therefore interested in whether or not there exists an $f$-stabilisation $\widetilde{f} = f+f'$ in the parameter space $X_f$ such that the following equation is satisfied for each vertex $i$: 
\begin{equation} c_l  \lambda_i  -  D_l=0. \label{fixedlambda} \end{equation} Each  entry in the above vector can be viewed as a polynomial in the entries of $f'$. Therefore \eqref{fixedlambda} cuts out a closed subvariety $X_{f,\alpha}^{i,\operatorname{lim}}$ of $X_f$ defined by the vanishing of the polynomial equations appearing in the entries of the vector in the left-hand side of \eqref{fixedlambda}. 
Let $$X_{f, \alpha}^{\operatorname{lim}} = \bigcap_{i} X_{f,\alpha}^{i,\operatorname{lim}} \subseteq X_f,$$ where the intersection ranges over all child vertices $j$ of $\mathcal{G}$. This is a closed subvariety of the parameter space $X_f$ of $f$-stabilisations, with defining equations given explicitly by \eqref{fixedlambda}. We have thus proved the following.

\begin{proposition}[When is a $\Lambda$-MLE given $f$ the limit $\Lambda$-MLE given an $f$-stabilisation?] \thlabel{answerq2}
Let $f$ denote a sample and $\lambda$ a $\Lambda$-MLE given $f$. Then $$ X_{f,\lambda}^{\operatorname{lim}} \subseteq X_f$$  parametrises those $f$-stabilisations such that the $\Lambda$-MLE given $\widetilde{f}(\epsilon) : = f+ \epsilon f'$ tends to $\lambda$ as $\epsilon$ tends to zero. In particular, the $\Lambda$-MLE $\lambda$ given $f$ is the limit $\Lambda$-MLE given an $f$-stabilisation if and only if $$ X_{f, \lambda}^{\operatorname{lim}} \neq \emptyset.$$ 
\end{proposition}

We can use \thref{answerq2} to answer the analogous question for MLEs rather than $\Lambda$-MLEs, as per \thref{answerq2gen} below which corresponds to \thref{secondmainresult}\ref{ps3}.  

\begin{corollary}[When is an MLE given $f$ the limit MLE given an $f$-stabilisation?] \thlabel{answerq2gen}
Assume an MLE $\alpha$ exists given sample $f$. 
Then $X_{f,\alpha}^{\operatorname{lim}} \subseteq X_f$ parameterises the $f$-stabilisations $\widetilde{f}$ such that the limit MLE given $\widetilde{f}$ is $\alpha$. In particular, $\alpha$ is a limit MLE given an $f$-stabilisation if and only if $X_{f, \alpha}^{\operatorname{lim}} \neq \emptyset$. 
\end{corollary} 

\begin{proof}
Let $\lambda$ denote the $\Lambda$-MLE component of $\alpha$.
By \thref{answerq2} we know that $\widetilde{f}$ lies in $X_{f, \lambda}^{\operatorname{lim}}$ if and only if its limit $\Lambda$-MLE is $\lambda$. But by \thref{mainresult}\ref{partc} we also know that the limit MLE is an MLE given $f$, as by assumption $f$ has at least one MLE. Since $\Omega$-MLEs are unique, there is a unique MLE given $f$ with a fixed $\Lambda$-MLE component. The MLE $\alpha$ has this property therefore the limit MLE given $\widetilde{f}$ is $\alpha$ as required. 
\end{proof}

\thref{answerq2gen} shows that $X_{f,\alpha}^{\operatorname{lim}}$ parametrises those $f$-stabilisations $\widetilde{f}$ in $X_f$ satisfying the property that the limit MLE given $\widetilde{f}$ is $\alpha$, which leads us naturally to the following

\begin{definition}[Parameter space of $f$-stabilisations with limit MLE $\alpha$] \thlabel{defofps2}
Let $f$ denote a sample and $\alpha$ an MLE given $f$. Then the closed subvariety $$X_{f, \alpha}^{\operatorname{lim}} \subseteq X_f$$ of the parameter space of $f$-stabilisations is the \emph{parameter space of $f$-stabilisations $\widetilde{f}$ such that $\alpha$ is the limit MLE given $\widetilde{f}$}. 
\end{definition}

\section{Linear regression} \label{sec:linearregression}

We illustrate our results for star-shaped graphs $\mathcal{G}$. These are connected graphs with a single child vertex, see Figure~\ref{fig:star}. Statistical models determined by graphs of this type are linear regression models: they express the child node as a linear combination of the parent nodes plus noise. In Section \ref{subsec:bounded} we consider the case where the MLE exists. In Section \ref{subsec:unbounded} we consider the case where the MLE does not exist.

\subsection{When the MLE exists} \label{subsec:bounded} 

We show that for a star-shaped $\Gcal$, if the MLE exists given a sample $f$ then the MLE given any $f$-stabilisation is the same: the minimal norm MLE given $f$. We apply results from Section \ref{sec:MLEofsamplestab} to prove this. 
First we show that the conditions given in \thref{OmegaMLEofstab} are satisfied for all $f$-stabilisations. These characterise when the MLE given an $f$-stabilisation is an MLE given $f$. We prove this in \thref{condalwayssatisfied}. Secondly we show that only one MLE given $f$ can be obtained in this way and describe it explicitly, see \thref{specialcasegraph} below. This gives an explicit description of the parameter spaces $X_{f, \alpha}$ from \thref{defofps}, for all samples $f$ and MLEs $\alpha$ given $f$, see \thref{reformulation}.

\begin{proposition}[The MLE given any $f$-stabilisation is an MLE given $f$] \thlabel{condalwayssatisfied}
Fix a star-shaped graph $\mathcal{G}$ on $m$ vertices. 
Let $f$ be a sample and $\widetilde{f}$ a stabilisation of $f$. Assume the MLE given $f$ exists. 
Then the MLE given $\widetilde{f}$ is an MLE given $f$.
\end{proposition}

\begin{proof}
Without loss of generality the unique child vertex is vertex $m$. Let $f'$ be any $f$-perturbation. To show that the MLE given $\widetilde{f} = f+ f'$ is an MLE given $f$, by \thref{OmegaMLEofstab} it suffices to show that $$v_m \in \langle v_i :i \to m \rangle \text{ and } \overline{ f_m} +  \overline{v_m} \in \langle f_i + v_i :i \to m \rangle.$$ We will starting by proving 
\begin{equation}
\langle   f_i + v_i :i \to m \rangle  = \langle f_i  :i \to m \rangle \oplus \langle v_i :i \to m \rangle, \label{equalityofsubspaces}
\end{equation} which implies that $\overline{ f_m} + \overline{v_m} \in \langle f_i + v_i :i \to m \rangle$, since $\overline{ f_m}$ and $\overline{v_m}$ lie in $\langle f_i :i \to m \rangle $ and $\langle v_i :i \to m \rangle$ respectively.

We have 
    \begin{equation}
        \langle   f_i + v_i :i \to m \rangle  \subseteq \langle f_i  :i \to m \rangle \oplus \langle v_i :i \to m \rangle \subseteq \langle f_i :i \to m \rangle \oplus \langle v_1, \hdots, v_m \rangle. \label{inclusions}
    \end{equation}
The left-hand side has dimension equal to the number of parents of $m$, namely $m-1$, since the rows of $\widetilde{f}$ are linearly independent. We now show that the right-hand side has dimension less than or equal to $m-1$. 

By definition of an $f$-perturbation, the map $f': \KK^m \to \KK^n$ has kernel of dimension $r: = \dim \im f$. Then the span of the columns $v_1, \hdots, v_m$ of $f'$ has dimension $m-r$, and the span of the columns $f_1, \hdots,  f_m$ of $f$ has dimension $r$. Since $f$ is semistable we know that $ f_m \notin \langle f_i :i \to m \rangle$. Therefore $\operatorname{dim} \langle f_i :i \to m \rangle = r-1$.  It follows that  the right-hand side of \eqref{inclusions} has dimension $m-1$.
As a result, the inclusions in \eqref{inclusions} above must all be equalities, giving \eqref{equalityofsubspaces}.

It remains to show that $v_m \in \langle v_i :i \to m \rangle$. In fact we will show the stronger statement that $v_m = 0$. Recall that $f'$ has kernel equal to $(\ker  f)^{\perp}.$ Therefore to show that $v_m = 0$, it suffices to show that the standard basis vector $e_m : = (0, \hdots, 0,1) \in \KK^m$ lies in $(\ker f)^{\perp} = \im f^T$, as $v_m = f'(e_m)$.  Since $f$ is semistable, we know that $ f_m \notin \langle f_i :i \to m \rangle$, so that $x: =  f_m - \overline{ f_m} \neq 0$. Note that $x \in \langle  f_1, \hdots,  f_m \rangle$. By construction $x \in \langle f_i  :i \to m \rangle^{\perp}$, therefore $f_i \cdot x = 0$ for all $ i \to m$.  Note also that $ f_m \cdot x \neq 0$ since otherwise $x \in \langle  f_1, \hdots,  f_m \rangle^{\perp} \cap \langle  f_1, \hdots,  f_m \rangle = \{0\}$ which contradicts $x \neq 0$. Therefore $f^T(x) = e_m$, so that $e_m \in \im  f^T$ as required. 
\end{proof}

We now strengthen \thref{condalwayssatisfied}. That is, in \thref{specialcasegraph} below we show that for a sample $f$ such that an MLE given $f$ exists, not only do we have that the MLEs given $\widetilde{f}$ are MLEs given $f$ 
for any $f$-stabilisation $\widetilde{f}$, but also that only one MLE given $f$ can be obtained in this way, namely the minimal norm MLE given $f$. This is \thref{linearregression}.  

\begin{proposition}[The MLE given any $f$-stabilisation is the minimal norm MLE given $f$]  \thlabel{specialcasegraph} 
Fix a star-shaped graph $\mathcal{G}$ on $m$ vertices and let $f$ denote a sample such that an MLE given $f$ exists. Then the $\Lambda$-MLE given any stabilisation $\widetilde{f}$ of $f$ is the minimal norm $\Lambda$-MLE given $f$. 
\end{proposition}

\thref{reformulation} below gives an explicit description of the parameter space $X_{f,\alpha}$ from Section \ref{subsec:whichMLEs}, for any sample $f$ for which $\alpha$ is an MLE given $f$. 

\begin{corollary} \thlabel{reformulation} 
Fix a connected DAG $\mathcal{G}$ on $m$ vertices with a unique child vertex, and let $f$ denote a sample such that an MLE given $f$ exists. Let $\alpha$ denote any MLE given $f$. Then 
$$X_f \supseteq X_{f, \alpha} = \begin{cases} \emptyset & \text{if $\alpha$ is not the minimal norm MLE given $f$;} \\
 X_{f} & \text{if $\alpha$ is the minimal norm MLE given $f$.} \end{cases} $$
\end{corollary} 

\begin{proof}[Proof of \thref{specialcasegraph}] By relabeling the vertices of $\mathcal{G}$ if necessary we can assume that $m$ is the unique child vertex.  Let $\widetilde{f} = f+f'$ denote a stabilisation of $f$. 

The MLE given $\widetilde{f}$ is an MLE given $f$, by \thref{condalwayssatisfied}. Moreover, in the proof of \thref{condalwayssatisfied} we have also shown that $v_m =0$, where $v_m$ is the last column of $f'$. To show that the MLE given $\widetilde{f}$ is the minimal norm MLE given $f$, recall that the $\Lambda$-MLE $\{\lambda_{im}\}_{i \to m}$ given $\widetilde{f}$ is determined by the equation $$ \overline{ f_m} + \overline{v_m} = \sum_{i \to m} \lambda_{im} f_i + \sum_{i \to m} \lambda_{im} v_i.$$ 

Since $\overline{v_m} = \overline{0} = 0$, the coefficients $\lambda_{im}$ satisfy $\sum_{i \to m} \lambda_{im} f_i = \overline{ f_m}$ and $\sum_{i \to m} \lambda_{im} v_i = 0$, using the fact that the $v_i$ and $f_i$ are orthogonal to each other. The latter equation is equivalent to asking that the vector $\lambda_m = (\lambda_{1m}, \lambda_{2m}, \hdots, \lambda_{m-1,m})$ lies in $\operatorname{ker} {f'_m}$ where $f'_m$ is obtained from $f'$ by removing the last column. Let $f_m$ denote the matrix obtained by removing the last column of $f$. Then the minimal norm MLE given $f$ has as its $\Lambda$-MLE the solution to the system $f_m  x = \overline{f_m}$ which lies in $(\ker  f_m )^{\perp}$. We claim now that $(\ker  f_m )^{\perp} = \ker  {f'_m} .$ 

By definition of a sample perturbation, we know that $(\ker f)^{\perp} = \ker f'$. Since $v_m = 0$, the rows of ${f'_m}$ and of $f_m$ are also orthogonal to each  other, therefore $(\ker f_m)^{\perp} \subseteq \ker f_m'$. To show that equality holds, we calculate the dimension of each side. Since $f_m \notin \langle f_i : i \to m \rangle$ by semistability of $f$, on the left-hand side we have $$\dim (\ker  f_m)^{\perp} = \dim \im f_m = \dim \im f - 1.$$ Since $v_m = 0$, we also have that $\dim \ker f_m' = \dim \ker f' - 1$. So on the right-hand side we have $$ \dim \ker f_m' = \dim \ker f' - 1 = \dim (\ker f)^{\perp} -1 = \dim \im f - 1.$$  Thus $(\ker f_m)^{\perp} = \ker f_m'$. Hence the MLE given $\widetilde{f}$ is the minimal norm MLE given~$f$. 
\end{proof}

For general DAG models we do not expect the above results to continue to hold for all samples $f$ such that an MLE given $f$ exists. It would be interesting to obtain counterexamples.

\subsection{When the MLE does not exist} \label{subsec:unbounded}

Since the $\Lambda$-MLE always exists, we study which $\Lambda$-MLEs can be achieved as the $\Lambda$-MLE given a stabilisation. We may also ask which $\Lambda$-MLEs can be achieved as the limit $\Lambda$-MLE given a stabilisation. We address both these questions through specific examples. 

\thref{unstableexample1} below gives an example of a DAG $\mathcal{G}$ and sample $f$ such that the $\Lambda$-MLE given any $f$-stabilisation is a $\Lambda$-MLE given $f$.  It also shows that any $\Lambda$-MLE given $f$ can be obtained as the $\Lambda$-MLE given an $f$-stabilisation. This is in contrast with \thref{specialcasegraph} in Section \ref{subsec:unbounded} above where only one $\Lambda$-MLE can be obtained. Finally, it provides an explicit description of the varieties $X_{f, \lambda}^{\operatorname{lim}}$ appearing in \thref{answerq2}. 

\begin{proposition}  \thlabel{unstableexample1} 
Let $\mathcal{G}$ denote the DAG $ 1 \to 3 \leftarrow 2$, and let $f$ denote the sample with first column $ f_1 =(1,0,\hdots,0)$ and zero second and third column. Then the $\Lambda$-MLE given any $f$-stabilisation is an MLE given $f$, and moreover any $\Lambda$-MLE $\lambda$ of $f$ can be achieved as the $\Lambda$-MLE given a suitable $f$-stabilisation. In addition, given a $\Lambda$-MLE $\lambda = (0,b)$ of $f$, we have: $$ X_{f,\lambda}^{\operatorname{lim}} = \{f' \in X_f :v_2 \cdot v_3 = b (v_2 \cdot v_2) \},$$ where $v_2$ and $v_3$ are the second and third columns of $f'$ respectively. 
\end{proposition} 

\begin{proof}
The $\Lambda$-MLEs given $f$ are pairs of the form $(0,b)$ for $b \in \KK$. We now show that the $\Lambda$-MLE given any $f$-stabilisation has this form. A map $f': \KK^m \to \KK^n$ with columns $v_1, v_2,v_3$ is an $f$-perturbation if and only if $v_1 = 0$ (to ensure $\im f' \subseteq (\im f)^\perp$), $v_2, v_3$ have zero first entry (to ensure that $(\ker f')^\perp \subseteq \ker f$) and are linearly independent (to give equality $(\ker f')^{\perp} = \ker f$). Choose such an $f$-perturbation $f'$. 

Then the $\Lambda$-MLE given $\widetilde{f} = f+f'$ is the pair $(c,d)$ such that
    $$ \pi_{\langle  f_1 + v_1,  f_2 + v_2 \rangle}( f_3 + v_3) =  c  f_1 + d v_2.$$ But $$ \pi_{\langle  f_1 + v_1,  f_2 + v_2 \rangle}(\overline{ f_3} + \overline{v_3})  = \pi_{\langle  f_1, v_2 \rangle}(\overline{v_3}) = \overline{v_3},$$ since $\overline{v_3} \in \langle v_2 \rangle$. Here $\overline{ f_3} = \pi_{\langle  f_1,  f_2 \rangle}( f_3) $ and $\overline{v_3} = \pi_{\langle v_1, v_2 \rangle}(v_3)$. 
Therefore  $  c  f_1 + d v_2 = \overline{v_3} \in \langle v_2 \rangle$. Since $v_2$ and $ f_1$ are orthogonal, it follows that $c=0$ and that $d v_2 = \overline{v_3}$. Therefore $d = v_2 \cdot v_3 / v_2 \cdot v_2$. In other words, the $\Lambda$-MLE given $\widetilde{f}$ is the pair $(0, v_2 \cdot v_3 / v_2 \cdot v_2)$, which is a well-defined $\Lambda$-MLE given $f$. Note that we could also obtain this result by showing instead that the condition of \thref{MLEofstab} holds, but the direct proof we have given also proves the second part of \thref{unstableexample1}. Indeed, given any $\Lambda$-MLE $(0,b)$ of $f$, we can always find an $f$-perturbation $f'$ such that $v_2 \cdot v_3 = b ( v_2 \cdot v_2)$. 

The equation above defines a quadratic $Q_b$ in $X_f$, and the $\Lambda$-MLE  given $\widetilde{f}$ is $(0,b)$ (which coincides with the limit $\Lambda$-MLE given $\widetilde{f}$) if and only if $\widetilde{f} \in Q_b$. Therefore setting $\lambda=(0,b)$, we have $$ X_{f,\lambda}^{\operatorname{lim}} = X_f \cap Q_b = \{ f' \in X_f :\langle v_2, v_3 \rangle = b \langle v_2, v_2 \rangle \}. \qedhere$$ 
\end{proof} 

We now give an example of a sample $f$ such that the $\Lambda$-MLE given any $f$-stabilisation is never a $\Lambda$-MLE given $f$. We use this example to illustrate \thref{mainresultLambda}, by describing the $\Lambda$-MLE given $\widetilde{f}(\epsilon) = f+ \epsilon f'$ for any $f$-stabilisation $\widetilde{f} = f +f'$ and its limit as $\epsilon \to 0$.

\begin{proposition} \thlabel{unstableexample2} 
Let $\mathcal{G}$ denote the DAG $1 \to 3 \leftarrow 2$. Let $f$ denote the sample with first column $ f_1= (1,0,\hdots, 0)$, second column $ f_2 = (1,1,0, \hdots, 0)$ and third column $ f_3 = (2,1,0,\hdots, 0)$, with unique $\Lambda$-MLE given $f$ equal to $(1,1)$. Then the $\Lambda$-MLE given any $f$-stabilisation is not $(1,1)$. Moreover, the $\Lambda$-MLE given $\widetilde{f}(\epsilon) = f + \epsilon f'$ for any $f$-stabilisation $\widetilde{f} = f+f'$ and $\epsilon \neq 0$ is $$ \left( \frac{1 - 2 \epsilon^2 (v \cdot v)}{1 + \epsilon^2 (v \cdot v)} , 1 \right),$$ which tends to $(1,1)$ as $\epsilon \to 0$. 
\end{proposition} 

\begin{proof} 
Let $f'$ denote an $f$-perturbation. Then $f'$ has columns $-v,-v,v$ for some non-zero $v$ 
with zero first and second entries. 
The $\Lambda$-MLE given $\widetilde{f} = f+ f'$ is the pair $(\alpha, \beta)$ such that $$ \pi_{\langle  f_1 -v,  f_2 - v \rangle} ( f_3 + v) = \alpha (  f_1 -v) + \beta(  f_2 - v).$$  We claim that the $\Lambda$-MLE given $\widetilde{f}$ is not $(1,1)$, the unique $\Lambda$-MLE given $f$. By \thref{MLEofstab}, this follows from the fact that  $ f_3 + v$ does not lie in $\langle  f_1 -v,  f_2 - v \rangle$. 

We can check directly that the $\Lambda$-MLE given $\widetilde{f}$ is not a $\Lambda$-MLE given $f$. To calculate the $\Lambda$-MLE given $\widetilde{f}$, observe that $ \langle   f_1 -v,  f_2 - v \rangle = \langle  f_1 - v, e_2 \rangle,$ where $e_2$ is the second standard basis vector in $\KK^n$. Since the vectors in the latter span are orthogonal, we have: \begin{align*} 
\pi_{\langle  f_1 -v,  f_2 - v \rangle} ( f_3 + v)  & = \pi_{\langle  f_1 - v, e_2 \rangle }( f_3 + v) \\
& = \frac{( f_1 - v) \cdot ( f_3 + v)}{( f_1 -v) \cdot ( f_1 - v)} ( f_1 - v) + \frac{e_2 \cdot ( f_3 + v) }{e_2 \cdot e_2} e_2 \\
& = \frac{2- v \cdot v}{1 + v \cdot  v} (e_1 - v) + e_2 \\
& = \frac{1- 2 v \cdot v}{1 + v \cdot v} (e_1 - v) + (e_1 + e_2 - v). 
\end{align*} 
Therefore the $\Lambda$-MLE given $\widetilde{f}$ is  $$ (\alpha, \beta) = \left( \frac{1 - 2 v \cdot v}{1 + v \cdot v} , 1 \right).$$
Since $v$ is non-zero, we have that $(\alpha, \beta) \neq (1,1)$ for any $\widetilde{f}$.
An analogous calculation to the one above shows that the $\Lambda$-MLE given $\widetilde{f}(\epsilon)$ is $$ \left( \frac{1 - 2 \epsilon^2 (v \cdot v)}{1 + \epsilon^2 (v \cdot v)} , 1 \right).$$ This expression, while never equal to the unique $\Lambda$-MLE $(1,1)$ given $f$ for $\epsilon \neq 0$, tends to $(1,1)$ as $\epsilon$ tends to zero. 
\end{proof}

The above results suggest the following open questions, which are open even for DAG models on star-shaped graphs.

\begin{question}
Given a DAG $\mathcal{G}$, can we characterise those samples $f$ such that the $\Lambda$-MLE given any $f$-stabilisation is a $\Lambda$-MLE given $f$? Is there a sample $f$ such that some $f$-stabilisations have as their $\Lambda$-MLE a $\Lambda$-MLE given $f$, but others don't? Is there an unstable sample $f$ and $\Lambda$-MLE $\lambda$ of $f$ such that $X_{f, \lambda}^{\operatorname{lim}}$ is empty or all of $X_f$? 
\end{question} 

Regarding the first question, \thref{condalwayssatisfied} shows that for DAG models on star-shaped graphs all samples $f$ such that an MLE given $f$ exists have this property, whilst \thref{unstableexample1} and \thref{unstableexample2} show that unstable samples may or may not have this property. We conjecture, based on these results, that for star-shaped graphs the $\Lambda$-MLE given any $f$-stabilisation of a sample $f$ is a $\Lambda$-MLE given $f$ either if a $\Lambda$-MLE given $f$ exists, or if $f$ does not admit any linear dependencies amongst the unique set of parents.

\section{Outlook} \label{sec:outlook}

This paper gives a way to package an affine lift of a complete collineation from $\PP(\KK^m)$ to $\PP(\KK^n)$ into a sample for a DAG model on $m$ vertices. The MLE given such a sample is unique.
In this section we consider how one might think of the
moduli space of complete collineations as a statistical model.
In such a model, samples should correspond to affine lifts of complete collineations and the MLE given any sample should be unique.
 
We describe a sampling algorithm that takes as input a usual sample and outputs a complete collineation in Section \ref{subsec:ccassamples}. In Section \ref{subsec:blowups} we ask which statistical models may have affine lifts of complete collineations as their sample space.

\subsection{Sampling complete collineations} 
\label{subsec:ccassamples}

We describe an algorithm for obtaining an affine lift $(f_1,\hdots,f_t)$ of a complete collineation from $\PP(\KK^m)$ to $\PP(\KK^n)$ with first term a sample $f: \KK^m \to \KK^n$. We assume $ n \geq m$, which is without loss of generality by Section~\ref{sec:relating}.  

If $f$ has full rank, then $([f ])$ is a complete collineation. If not, choose a basis for $\ker f$, which consists of vectors that are linear combinations of the $m$ variables. We then sample each vector in this basis a total of $\dim \coker f = n - \dim f$ times. Sampling along a linear combination of variables appears in data analysis contexts such as \cite{squires2023linear}. We do not allow the case where all samples obtained from this procedure are zero. Consider the $\dim \coker f \times \dim \ker f$ matrix whose columns are these samples. By identifying $\coker f$ with $(\im f)^{\perp}$ via the standard inner product on $\KK^n$, and choosing a basis for $(\im f)^{\perp}$, this matrix determines a map $f_2: \ker f \to (\im f)^{\perp} \cong \coker f$. If $f_2$ has maximal rank, then $([f],[f_2])$ is a complete collineation, and we stop. If not, we follow the same procedure, replacing $f $ by $f_2$. Eventually, we reach $f_t$ of maximal rank, thus giving the desired affine lift $(f ,f_2,\hdots, f_t)$. 

It is important to observe the distinction in the choice of basis for $\ker f_i$ compared to $(\im f_i)^{\perp}$. Indeed, the choice of the former influences the sampling itself, since we sample linear combinations of the vertices corresponding to the chosen basis vectors of $\ker f$. By contrast, the sampling is independent of the choice of basis of $(\im f_i)^{\perp}$, depending only on $\operatorname{dim} (\im f_i)^{\perp}$. It is unclear how the sampling procedure could be modified so that it changes according to the basis chosen for $(\operatorname{im} f_i)^{\perp}$ at each stage, thereby making the algorithm canonical.

This question may be better answered from a different perspective, by thinking about what statistical model might have the moduli space of complete collineations as its space of samples. Section \ref{subsec:blowups} explores this perspective.

\subsection{Complete collineations as the sample space for a statistical model} \label{subsec:blowups}

Our hope is that there should exist a statistical model determined by a DAG $\mathcal{G}$ on $m$ nodes, such that a sample for this statistical model corresponds to an affine lift of a complete collineation from $\PP(\KK^m)$ to $\PP(\KK^n)$ for some $n$. 
Gaussian group models~\cite{amendola2021invariant} offer a promising starting point, if we assume that $\mathcal{G}$ is transitive. In this case the DAG model on $\mathcal{G}$ 
coincides with the Gaussian group model determined by the representation of the group $G(\mathcal{G})$ on $\KK^m$, where $$ G(\mathcal{G}) := \{a \in \GL_m(\KK) \ | \ a_{ij} = 0 \text{ for all $i \neq j$ such that $j \not\to i$ in $\mathcal{G}$}\}.$$
The representation of $G(\mathcal{G})$ on $\KK^m$ naturally extends to a representation on $(\KK^m)^n$ for any~$n$. This is the right multiplication action of $G(\mathcal{G})$ on $\Mat_{n \times m}(\KK)$, which induces a right multiplication action on $\PP(\Hom(\KK^m,\KK^n))$. Since $\mathcal{M}$ is a blow-up of $\PP(\Hom(\KK^m,\KK^n))$ along $G(\mathcal{G})$-invariant centres, it has an induced action of $G(\mathcal{G})$. By contrast to $\PP(\Hom(\KK^m,\KK^n))$, the moduli space $\mathcal{M}$ is not of the form $\PP(U)$ with the $G(\mathcal{G})$ action induced by a representation $G(\mathcal{G}) \to \GL(U)$, so it is not obvious how to associate to the action of $G(\mathcal{G})$ on $\mathcal{M}$ a Gaussian group model. 

One approach is to use the fact that the action of $G(\mathcal{G})$ on $\mathcal{M}$ is linear, so that $\mathcal{M}$ can be embedded $G(\mathcal{G})$-equivariantly inside a larger projective space $\PP(U)$, with $G(\mathcal{G})$ acting linearly on $\PP(U)$ via a representation $G(\mathcal{G}) \to \GL(U)$. This representation does indeed gives rise to a Gaussian group model. Unfortunately, this is not quite the model we are after. The sample space is too big: we are interested only in the subvariety of those samples corresponding to complete collineations, and it is unclear how to interpret the condition that $[f] \in \PP(U)$ lies in $\mathcal{M}$ in a statistically meaningful way. Moreover, it is unclear how to relate MLEs for this new model to MLEs for the original model. 

A remaining open problem then is whether there is another statistical model that can be constructed from the action of $G(\mathcal{G})$ on $\mathcal{M}$, one in which samples are affine lifts of complete collineations, MLEs given samples are always unique, and MLEs can be more easily related to those of the DAG model on $\mathcal{G}$.

\bibliographystyle{alpha}
 \bibliography{literatur}
 
\end{document}